\documentclass[sn-mathphys]{sn-jnl}

\usepackage{amsmath}
\usepackage{amssymb}
\usepackage{bm}

\usepackage{mathabx}

\jyear{2023}
\theoremstyle{thmstyleone}%
\newtheorem{theorem}{Theorem}
\newtheorem{lemma}[theorem]{Lemma}%

\theoremstyle{thmstyletwo}%

\theoremstyle{thmstylethree}%

\DeclareMathOperator{\spn}{span}

\newcommand{\BR}{\mathbb R}
\newcommand{\Omegax}{\Omega^{(1)}}
\newcommand{\Omegav}{\Omega^{(2)}}
\newcommand{\Omegai}{\Omega^{(i)}}
\newcommand{\Vx}{V^{(1)}}
\newcommand{\Vv}{V^{(2)}}
\newcommand{\Vi}{V^{(i)}}
\newcommand{\Gammar}{\widehat\Gamma^-}
\newcommand{\Gammarm}{\widehat\Gamma^-}
\newcommand{\Gammarp}{\widehat\Gamma^+}
\newcommand{\Gammarpm}{\widehat\Gamma^\pm}
\newcommand{\Gammax}{\widehat\Gamma^-_x}

\DeclareBoldMathCommand{\bbeta}{\beta}
\DeclareBoldMathCommand{\balpha}{\alpha}
\DeclareBoldMathCommand{\br}{r}
\DeclareBoldMathCommand{\bk}{k}
\DeclareBoldMathCommand{\bb}{b}
\DeclareBoldMathCommand{\bv}{v}
\DeclareBoldMathCommand{\bx}{x}
\DeclareBoldMathCommand{\bu}{u}
\DeclareBoldMathCommand{\bj}{j}
\DeclareBoldMathCommand{\bs}{s}
\DeclareBoldMathCommand{\bl}{l}
\DeclareBoldMathCommand{\bell}{\ell}
\DeclareBoldMathCommand{\be}{e}
\DeclareBoldMathCommand{\bn}{n}
\DeclareBoldMathCommand{\bE}{E}
\DeclareBoldMathCommand{\bff}{f}
\DeclareBoldMathCommand{\bone}{1}

\begin{document}

\title[Sparse Grid DG Method with SD for Transport Equations]{Sparse Grid Time-Discontinuous Galerkin Method with Streamline Diffusion for Transport Equations}

\author*[1]{\fnm{Andreas} \sur{Zeiser}}\email{andreas.zeiser@htw-berlin.de}

\affil*[1]{\orgdiv{Fachbereich 1 Ingenieurwissenschaften -- Energie und Information}, \orgname{HTW Berlin}, \orgaddress{\street{Wilhelminenhofstr. 75A}, \city{Berlin}, \postcode{12459}, \country{Berlin}}}

\abstract{High-dimensional transport equations frequently occur in science and engineering. Computing their numerical solution, however, is challenging due to its high dimensionality. 

In this work we develop an  algorithm to efficiently solve the transport equation in moderately complex geometrical domains using a Galerkin method stabilized by streamline diffusion. The ansatz spaces are a tensor product of a sparse grid in space and discontinuous piecewise polynomials in time. Here, the sparse grid is constructed upon nested multilevel finite element spaces to provide geometric flexibility. This results in an implicit time-stepping scheme which we prove to be stable and convergent. If the solution has additional mixed regularity, the convergence of a $2d$-dimensional problem equals that of a $d$-dimensional one up to logarithmic factors.

For the implementation, we rely on the representation of sparse grids as a sum of anisotropic full grid spaces. This enables us to store the functions and to carry out the computations on a sequence regular full grids exploiting the tensor product structure of the ansatz spaces. In this way existing finite element libraries and GPU acceleration can be used. The combination technique is used as a preconditioner for an iterative scheme to solve the transport equation on the sequence of time strips. 

Numerical tests show that the method works well for problems in up to six dimensions. Finally, the method is also used as a building block to solve nonlinear Vlasov-Poisson equations.}

\keywords{sparse grid, high-dimensional transport equations, streamline diffusion, combination technique, Vlasov-Poisson equations}

\pacs[MSC Classification]{65M60, 65M12}

\maketitle

\section{Introduction}

High-dimensional transport equations frequently arise  in science and engineering. In this paper we will consider the linear equation
\begin{align} \label{eq:transport}
\begin{split}
&\partial_t u + \bbeta(t,\br) \cdot \nabla_{\br} u + \sigma(t,\br) \, u = f(t,\br), \quad \textnormal{in } 
(0,T) \times \Omega  \\
& u(0,\br) = u_0(\br)
\end{split}
\end{align}
subject to the inflow boundary condition
\begin{align} \label{eq:inflow}
u=g \quad \textnormal{on } \Gamma_\br^- = \big\{(t,\br)\in(0,T)\times \partial\Omega\,\big\vert \, \bbeta(t,\br)\cdot \bn(\br) < 0 \big\}
\end{align}
where $\bn(\br)$ is the outward unit normal of the boundary, and $\bbeta,\,f, \,\sigma,\, g,\, u_0$ are smooth enough functions. We will focus on the case where the variable $\br$ can be partitioned into $\br = (\bx, \bv)$ and $\bx, \bv$ are $d$-dimensional variables with $d=1,2,3$ (i.e.,\ space and velocity). We assume that $\Omega = \Omegax \times \Omegav$ is a product domain with $\Omegax, \Omegav \subset \BR^d$, i.e.\
\[ u: (0,T) \times \Omegax \times \Omegav \rightarrow \mathbb R \]
is a function in time and two variables with dimension $d$.

Simulations of such systems are computationally demanding since the time evolution of an up to six-dimensional function has to be  calculated. Applying standard discretization schemes leads to an evolution equation in $\mathcal O(n^{2d})$ degrees of freedom, where $n$ is the number of grid points in one dimension and a corresponding high computational effort. To tackle the problem, methods such as  particle methods \cite{verboncoeur2005},  adaptive multiscale methods \cite{besse2008,deriaz2018} and tensor product methods \cite{einkemmer2018a} have been used. 

Sparse grids \cite{bungartz2004} are a means of overcoming the curse of dimensionality and have been applied to transport equations \cite{schwab2008} as well as to kinetic equations: in \cite{bokanowski2013a,kormann2016a} interpolation on sparse grids is used in a semi-Lagrangian method while \cite{guo2016} uses sparse grids with discontinuous ansatz functions. In these approaches tensor products of one-dimensional bases are used, restricting the domain to rectangular regions. However, finite element spaces can also be used in the construction of sparse grid spaces \cite{griebel2013,harbrecht2008} allowing for greater geometric flexibility. 

The related problem of radiative transfer has been studied widely. In \cite{kanschat1998} standard finite element discretization has been used to solve the stabilized equation. To cope with the high dimensionality, sparse grids of finite element spaces  \cite{widmer2008}, sparse tensor product of multilevel finite element spaces and spherical harmonics  \cite{grella2011} as well as a corresponding sparse grid combination technique \cite{grella2011a} have been used. More recently adaptive schemes are applied to solve the parameter dependent transport equation \cite{dahmen2018,dahmen2019}.

Our aim is to use sparse grid spaces based upon finite element spaces to devise a stable, convergent and efficient method for the solution of kinetic equations in moderately complex geometrical domains. We also focus on their efficient implementation using well-established finite element libraries and fast algorithms. 

For the discretization of space continuous piecewise polynomial finite element spaces for each domain $\Omega^{(i)}$ are used. Compared to discontinuous elements less degrees of freedom are needed to represent these functions. This is especially important as this factor enters quadratically in the degrees of freedom of the sparse grid space. 

The time domain is partitioned and on each time strip space-time elements are constructed by tensorizing polynomials in time with sparse grids in space. Using this space in a Galerkin method results in transport equations on each time strip, where the result of the previous step enters the current step as an initial condition. Enforcing this condition weakly leads to discontinuous functions in time.

However, it is known that the finite element method is unstable if the solution has discontinuities and hence has to be stabilized. One option is to use the streamline diffusion method \cite{johnson1984}. There, an artificial diffusion is added in the direction of the streamline leading to a stable scheme, while still keeping the order of the method.

This method can be interpreted as an implicit time-stepping scheme, where stabilized transport equations have to be solved on each time strip. For that purpose matrix-free iterative methods are applied. For an efficient calculation the sparse grid functions are represented as a linear combination of functions in anisotropic full grid spaces, which is also used in the combination technique (see \cite{garcke2012} and references therein). This leads to regular data structures and fast applications of the discretization matrices by exploiting the tensor product structure. Finally, the combination technique will serve as a preconditioner for the iterative solution of the corresponding equations \cite{griebel1994}. 

We will proceed along the following lines. In Sect. \ref{sec:discretization} sparse grid spaces are introduced, and the weak form, including streamline diffusion, is employed to derive discrete equations for each time step. This algorithm is analyzed in Sect. \ref{sec:analysis} in the case of constant coefficients with respect to stability and convergence. Sect. \ref{sec:algorithm} describes the algorithm for efficient computations based on the tensor product structure and multilevel spaces. The algorithm is applied to test cases in Sect. \ref{sec:experiments} as well as to Vlasov-Poisson equations in one and two spatial dimensions.

In the following we denote by $C \lesssim D$ that $C$ is bounded by a multiple of $D$ independently of the level of the sparse grid. $C \sim D$ is defined as $C \lesssim D$ and $D \lesssim C$.

\section{Discretization and Weak Formulation\label{sec:discretization}}

In this section we introduce the weak formulation of the transport problem including streamline diffusion \cite{johnson1984} for stabilization. For the discretization we will use a sparse grid based on finite elements and discontinuous polynomials in time. Finite element spaces have been used in the construction of sparse grids, realized either by wavelet-type \cite{schwab2003,widmer2008} or multilevel frames \cite{harbrecht2008}. We will represent sparse grids as in \cite{griebel2014}. This enables us to use classical nested finite element spaces. We partition the time domain, use discontinuous piecewise polynomials and a full tensor product with the sparse grid, see \cite{hilber2009} in the case of parabolic equations. Eventually this will lead us to an implicit time-stepping scheme.

Let $\Omegax$ and $\Omegav$ be two open polyhedral Lipschitz domains in $d$ dimensions and  $\Vi_1$ conforming finite element spaces on $\Omegai$ of piecewise polynomial ansatz functions. Define nested multilevel finite element spaces 
\begin{align*}
\Vi_1 \subset \Vi_2 \subset \ldots \subset H^1(\Omegai), \quad i=1,2
\end{align*}
by uniform refinement such that $\dim \Vi_\ell \sim 2^{d\ell}$. Based on these spaces the sparse tensor product space of level $L=1,2,\ldots$
\begin{align} \label{eq:sg_space}
V_L = \sum_{\vert \bell\vert_1 =  L + 1} \Vx_{\ell_1} \otimes \Vv_{\ell_2}, \quad  \bell = [\ell_1,\ell_2], \, \vert \bell\vert_1 = \ell_1+\ell_2
\end{align}
is constructed with  $\dim V_L \lesssim L 2^{d L}$ \cite{griebel2014}. Up to a logarithmic factor the number of degrees of freedom scales like a standard globally refined finite element discretization of a $d$-dimensional domain.

In the time domain $J=(0,T)$ define $0 = t_0 < t_1 < \ldots < t_N = T$ and the partition 
\[ \mathcal I = \{ I_j \}_{j=1,\ldots,N},\quad I_j = (t_{j-1},t_j), \quad \vert \mathcal I\vert  = \max_{j=1,\ldots,N}(t_j - t_{j-1}). \]
The discrete function space is defined as the set of all piecewise polynomials of order $r$ with coefficients in $V_L$, i.e.
\[ V_{L,\mathcal I} = \big\{ v\in L^2\big( J; V_L\big) \,\big\vert \, v\vert_I \in \mathbb P^r(I; V_L) \textnormal{ for all } I\in \mathcal I \big\}, \]
where $\mathbb P^r$ denotes the polynomials degree $r$. 

We will use the space $V_{L,\mathcal I}$ in a Galerkin scheme to compute an approximate solution of \eqref{eq:transport} and \eqref{eq:inflow}. 

For that purpose, we use the streamline diffusion method on each time strip $I_j$ \cite{johnson1984}. The initial condition of \eqref{eq:transport} and the inflow condition \eqref{eq:inflow} are enforced in a weak sense: 

For $j=1,2,\ldots,N$ find a function $U_j \in \mathbb P^r(I_j, V_L)$ on the time strip $I_j \in \mathcal I$, $j=1,\ldots,N$, such that
\begin{align} \label{eq:weak_form_j}
a_j^{(\delta)}(U_j,w) = b_j^{(\delta)}(w), \quad \textnormal{for all } w \in \mathbb P^r(I_j, V_L)
\end{align}
where $\delta > 0$ is the parameter for the streamline diffusion and
\begin{align*}
\begin{split}
a_j^{(\delta)}(v,w) &= \int_{I_j} \big(\partial_t v + \partial_\bbeta v + \sigma v, w + \delta (\partial_t w + \partial_\bbeta w)\big) 
		- \langle v, w\rangle_{\Gammar(t)} \, \mathrm d t \\
	& \quad + (v_{j-1}^+, w_{j-1}^+) \\
b_j^{(\delta)}(w) &= \int_{I_j} \big(f, w + \delta (\partial_t w + \partial_\bbeta w)\big)
		- \langle g, w\rangle_{\Gammar(t)} \, \mathrm d t
	+(U_{j-1}^-, w_{j-1}^+).
\end{split}
\end{align*}
Here $\partial_\bbeta u = \bbeta \cdot \nabla u$, $u_j^\pm = \lim_{t\rightarrow t_j^\pm} u(t, \cdot)$, $U_0^- = u_0$. By $(\cdot,\cdot)$ and $\|\cdot\|$ denote the inner product and norm on $L^2(\Omega)$, respectively. For the boundary define
\begin{align*} 
\Gammarpm(t) = \big\{\br\in\partial\Omega \,\big\vert \, \bbeta(t,\br) \cdot \bn(\br) \gtrless 0 \big\}
\end{align*} 
and for any $\widehat \Gamma \subset \partial \Omega$
\begin{align*}
\langle u,v\rangle_{\widehat \Gamma} = \int_{\widehat \Gamma} (\bbeta \cdot \bn) \, u v \, \mathrm d \bs, \quad \vert u\vert_{\widehat \Gamma}^2 = \int_{\widehat \Gamma} \vert \bbeta\cdot \bn\vert  u^2 \, \mathrm d \bs. 
\end{align*}

Equation \eqref{eq:weak_form_j} can be solved successively, resulting in an implicit time-stepping procedure. Finally, the approximate solution $U \in V_{L,\mathcal I}$ is composed of the solutions $U_j$ with potential jumps at $t_j$.

\section{Analysis -- Constant Coefficients \label{sec:analysis}}

In this section the stability and convergence of the discretization is analyzed for the case of constant coefficients $\bbeta$ and $\sigma$. Hence, for the rest of this section we assume $\bbeta \in \BR^{2d}$, $\sigma \ge 0$, $f\in L^2(J\times \Omega)$ and $g\in L^2(J \times \Gammarm)$. A similar analysis for stationary transport dominated problems has been carried out in \cite{schwab2008} and for parabolic equations in \cite{hilber2009}. For an introduction to discontinuous time-stepping schemes and finite element spaces see \cite{larsson2005}.

However, these results have to be adapted to the current setting. With respect to stability, the weak formulation of the boundary condition has to be taken into account. To show convergence, classical approximation results for piecewise polynomial ansatz functions have to be combined with sparse grid estimates to handle the tensor product ansatz functions.

For the analysis we formulate the problem on the space $V_{L,\mathcal I}$ by summing up \eqref{eq:weak_form_j} for $j=1,\ldots,N$. This gives 
\begin{align} \label{eq:weak_eqn}
A_{\mathcal I}^{(\delta)}(U,w) = B_{\mathcal I}^{(\delta)}(u_0,f,g; w) \quad \textnormal{for all } \, w \in V_{L,\mathcal I},
\end{align}
where 
\begin{align}\label{eq:AI}
\begin{split}
A_{\mathcal I}^{(\delta)}(v,w) &= \sum_{j=1}^N \int_{I_j} \big( \partial_t v + \partial_\bbeta v + \sigma v, w + \delta ( \partial_t w + \partial_\bbeta w) \big) 
		- \langle v, w\rangle_{\Gammarm} \, \mathrm d t \\
& \quad + \sum_{j=1}^{N-1} ([v]_j, w_j^+) + (v_0^+, w_0^+)
\end{split} \\
B_{\mathcal I}^{(\delta)}(u_0,f,g; w) 
	&= \sum_{j=1}^N \int_{I_j} \big( f, w + \delta ( \partial_t w + \partial_\bbeta w) \big) -  \langle g, w\rangle_{\Gammarm} \, \mathrm d t 
	+ \big( u_0, w_{0}^+ \big) \label{eq:BI}
\end{align}
with the jump term $[v]_j = v_j^+ - v_j^-$. Note that the boundaries $\Gammarpm$ are now time independent.

For the analysis the streamline diffusion norm  
\begin{align} \label{eq:norm_SD}
\begin{split}
 \vvvert v \vvvert_{\delta}^2 &= 
	\sum_{j=1}^N \int_{I_j} \sigma \|v\|^2 + \delta \| \partial_t v + \partial_\bbeta v\|^2 
		+ \vert v\vert ^2_{\partial\Omega} \, \mathrm d t  \\
	&+ \sum_{j=1}^{N-1} \| [v]_j \|^2  + \| v_0^+\|^2 + \| v_N^-\|^2 	
\end{split}
\end{align}
will play a central role. This norm gives extra control of the variations in the direction of the streamline.

\subsection{Stability}

In the first step we show that the method is stable with respect to the streamline diffusion norm. The following equivalence plays a central role.

\begin{lemma} \label{lem:Bvv_SD}
Let $v \in H^1(J \times \Omega)$, $\sigma \ge 0$, $0 < \delta$ and $\delta<1/\sigma$ in the case $\sigma > 0$. Then
\begin{align} \label{eq:norm_equivalence}
\frac 1 2 \vvvert v \vvvert_{\delta}^2 \le A_{\mathcal I}^{(\delta)}(v,v) \le \frac 3 2 \vvvert v \vvvert_{\delta}^2.
\end{align}
\end{lemma}
\begin{proof}
Applying Green's identity 
\begin{align*}
\int_{I_j} (\partial_t v, w) \, \mathrm d t
	&=  \big(v_j^-, w_j^-\big) - \big(v_{j-1}^+,w_{j-1}^+\big) 
		-\int_{I_j} \big( v,  \partial_t w \big) \, \mathrm dt \\
\int_{I_j} (\partial_\bbeta v, w) \, \mathrm d t
	&= \int_{I_j} \langle v, w\rangle_{\partial \Omega} - (v, \partial_\bbeta w) \, \mathrm d t
\end{align*}
on \eqref{eq:AI} gives
\begin{align} \label{eq:AI_partint}
\begin{split}
A_{\mathcal I}^{(\delta)}(v,w)
	 &= \sum_{j=1}^N \int_{I_j} (-v, \partial_t w + \partial_\bbeta w) + (\sigma v, w + \delta(\partial_t w + \partial_\bbeta w)) \,  \mathrm d t \\
& \quad + \sum_{j=1}^N \int_{I_j}\delta (\partial_t v + \partial_\bbeta v, \partial_t w + \partial_\bbeta w) + \langle v, w\rangle_{\Gammarp} \,  \mathrm d t \\
&\quad	 +  \sum_{j=1}^{N-1} (v_j^-, -[w]_j) + (v_N^-,w_N^-).
\end{split}
\end{align}
Averaging the two forms for $A_{\mathcal I}^{(\delta)}(v,v)$ gives
\begin{align} \label{eq:Avv}
\begin{split}
A_{\mathcal I}^{(\delta)}(v,v) &= 
	\sum_{j=1}^N \int_{I_j}  \sigma \|v\|^2 + \frac 1 2 \vert v\vert^2_{\partial\Omega} 
		+ \delta \| \partial_t v + \partial_\bbeta v\|^2 \, \mathrm d t 
		 \\
& +  \sum_{j=1}^N \int_{I_j}  \delta \sigma (v, \partial_t v + \partial_\bbeta v) \, \mathrm d t + \frac 1 2  \sum_{j=1}^{N-1} \| [v]_j \|^2 
		+ \frac 1 2  \|v_0^+\|^2 + \frac 1 2 \| v_N^-\|^2.
\end{split}
\end{align}
For $\sigma = 0$ 
\begin{align*}
A_{\mathcal I}^{(\delta)}(v,v) &= 
	\sum_{j=1}^N \int_{I_j}  \frac 1 2 \vert v\vert^2_{\partial\Omega} 
		+ \delta \| \partial_t v + \partial_\bbeta v\|^2 \, \mathrm d t 
		 	+ \frac 1 2  \sum_{j=1}^{N-1} \| [v]_j \|^2 
				+ \frac 1 2  \|v_0^+\|^2 + \frac 1 2 \| v_N^-\|^2,
\end{align*}
hence \eqref{eq:norm_equivalence} holds. In the case  $\sigma > 0$ the second term can be estimated by
\begin{align*}
\Big\vert  \sum_{j=1}^N \int_{I_j} \delta \sigma (v, \partial_t v + \partial_\bbeta v) \, \mathrm d t  \Big\vert 
\le \sum_{j=1}^N \int_{I_j} \frac{\delta \sigma^2}{2} \|v\|^2 + \frac{\delta}{2} \| \partial_t v + \partial_\bbeta v\|^2\, \mathrm d t. 
\end{align*}
Using $\delta < 1 / \sigma$, the equivalence follows directly using the definition of the streamline diffusion norm \eqref{eq:norm_SD}. 
\end{proof}

Now we use the Lemma to prove stability.
\begin{theorem} 
If $\delta>0$ and $\delta < 1/\sigma$ in the case $\sigma >0$, the weak form of the streamline diffusion equation with constant coefficients \eqref{eq:weak_eqn} has a unique solution $U$. If $\sigma > 0$ the system is stable in the sense that the solution satisfies
\begin{align*}
\vvvert U \vvvert_{\delta}^2 \lesssim 
 \sum_{j=1}^N \int_{I_j} \|f\|^2 + \vert g\vert ^2_{\Gammarm} \, \mathrm d t + \|u_0\|^2.
\end{align*}
For the case $\sigma = 0$ a similar bound can be derived where, however, the constant depends exponentially on the final time $T$. 
\end{theorem}
\begin{proof}
For $\sigma > 0$ applying Cauchy and Young's inequality on \eqref{eq:BI} gives
\begin{align*}
& B^{(\delta)}(u_0, f,g; v) \\ 
&\quad	\le \sum_{j=1}^N \int_{I_j} (\frac{1}{\sigma} + \delta) \|f\|^2 
		+ \frac{\sigma}{4} \|v\|^2 + \frac{\delta}{4} \|\partial_t v + \partial_\bbeta v\|^2
		+ \vert g\vert^2_{\Gammax} + \frac{1}{4} \vert v\vert^2_{\Gammarm} \, \mathrm d t \\
&\quad \quad
	    + \|u_0\|^2 + \frac{1}{4} \|v_0^+\|^2 \\
&\quad \le \frac 1 2 A_{\mathcal I}^{(\delta)}(v,v) + \sum_{j=1}^N \int_{I_j} (\frac{1}{\sigma} + \delta) \|f\|^2 
		+ \vert g\vert^2_{\Gammarm}  \, \mathrm d t + \|u_0\|^2, 
\end{align*}
where Lemma \ref{lem:Bvv_SD} was used in the last step. The solution $U$ satisfies
\begin{align*}
 A_{\mathcal I}^{(\delta)}(U,U) &= B_{\mathcal I}^{(\delta)}(u_0, f,g; U) \\
&\le \frac 1 2 A_{\mathcal I}^{(\delta)}(U,U) + \sum_{j=1}^N \int_{I_j} (\frac{1}{\sigma} + \delta) \|f\|^2 
		+ \vert g\vert^2_{\Gammarm}  \, \mathrm d t + \|u_0\|^2.
\end{align*}
Bringing $A_{\mathcal I}^{(\delta)}(U,U)$ to the left side and using \eqref{eq:norm_equivalence} gives the inequality which also shows the existence and uniqueness of the solution. 

For the case $\sigma =0$ one can introduce a positive constant term by change of the unknown to $w(t,\cdot) = \mathrm e^{-\alpha t} \, u(t,\cdot)$ which leads to an exponential factor $\mathrm e^{\alpha T}$ in the bound \cite{johnson1984}. 
\end{proof}

\subsection{Approximation}

In the following we will derive error estimates for the orthogonal projection onto $V_{L,\mathcal I}$, which will be used in the convergence proof in Sect.\ \ref{sub:convergence}. For the result we will combine estimates for sparse grids and well known approximation results from one-dimensional finite element discretization for the time domain. 

Following \cite{griebel2014} we assume that the approximation property 
\begin{align*}
\inf_{v_\ell \in \Vi_\ell} \| u - v_\ell \|_{H^q(\Omegai)} 
	\lesssim h_\ell^{s-q} \|u\|_{H^s(\Omegai)}, \quad u \in H^s(\Omegai), \quad i=1,2 
\end{align*}
holds for $q<\gamma$, $q\le s \le r+1$ uniformly in the level $\ell$  (note that we replaced $r$ by $r+1$). Here $h_\ell = 2^{-\ell}$ and 
\begin{align*}
\gamma = \sup\big\{s\in \mathbb R\,\vert \, \Vx_\ell \subset H^s(\Omegax), \, \Vv_\ell \subset H^s(\Omegav) \big\}.
\end{align*} 
Define for $s>0$ the Sobolev spaces of mixed order \cite{griebel2000}
\begin{align*}
H^{s}_{\mathrm{mix}}(\Omega) = H^{s,s}_{\mathrm{mix}}(\Omega), \quad
	H^{s_1,s_2}_{\mathrm{mix}}(\Omega)= H^{s_1}(\Omega^{(1)}) \otimes H^{s_2}(\Omega^{(2)}) 
\end{align*}
and denote the corresponding norm by $\|\cdot\|_{H^{s}_{\mathrm{mix}}}$. Then the following approximation results hold.
\begin{theorem}[\cite{griebel2014}] \label{thm:griebel}
Let $0<s\le r$ and denote by $P_L$ the $L^2$-orthogonal projection onto $V_L$. Then 
\begin{align*}
\| u- P_L u\|_{L^2(\Omega)} &\lesssim 2^{-(s+1) L} \sqrt{L}\| u \|_{H^{s+1}_\mathrm{mix}(\Omega)} \\
\| u- P_L u\|_{H^1(\Omega)} &\lesssim 2^{-s L} \sqrt{L} \| u \|_{H^{s+1}_\mathrm{mix}(\Omega)}
\end{align*}
if $u \in H^{s+1}_{\mathrm{mix}}(\Omega)$.
\end{theorem}
\begin{proof}
The inequalities directly follow from \cite[Thm.\ 1]{griebel2014} for the case $n_1=n_2=d$, $\sigma=1$,  and Examples 1 and 2. 
\end{proof}

Finally, define the space 
\begin{align*}
 \mathcal H^s(J\times \Omega) = H^s(J) \otimes L^2(\Omega) \cap L^2(J) \otimes H^{s}_{\mathrm{mix}}(\Omega)
\end{align*}
with classical regularity with respect to time and a mixed regularity with respect to $\bx$ and $\bv$. The corresponding norm reads
\begin{align*}
\|u\|_{\mathcal H^s}^2 = \int_0^T \|\partial_t^s u\|_{L^2(\Omega)}^2 + \|u\|_{H^{s}_{\mathrm{mix}}(\Omega)}^2 \, \mathrm d t.
\end{align*}
With this type of regularity we can show the following approximation result.

\begin{theorem} \label{thm:approximation}
Let $L\in \mathbb N$ and let $\mathcal I$ be a time partition such that $\vert \mathcal I\vert  \sim h = 2^{-L}$. For $u \in L^2(J\times \Omega)$ define 
\[ u_h = (P_{\mathcal I} \otimes P_L)\, u, \]
where $P_{\mathcal I}$ is the $L^2$-orthogonal projection onto the continuous piecewise polynomials of degree $r$ with respect to the partition $\mathcal I$. For $0 < s \le r$ and $u\in \mathcal H^{s+1}$
\begin{align*}
\| u - u_h \|_{L^2(J\times \Omega)}  
	&\lesssim 2^{-(s+1)L} \sqrt{L} \, \| u\|_{\mathcal H^{s+1}} \sim h^{s+1} \, (\ln h)^{1/2} \, \| u\|_{\mathcal H^{s+1}}\,,\\
\| u - u_h \|_{H^1(J\times \Omega)}  	
	& \lesssim 2^{-s L} \sqrt{L} \, \| u\|_{\mathcal H^{s+1}} \sim h^{s} \, (\ln h)^{1/2} \, \| u\|_{\mathcal H^{s+1}}\,,
\end{align*} 
where the constants are independent of $L$ and $u$.
\end{theorem}
\begin{proof}
For the proof we will use norms of the projection operators. For the time domain, the operator $P_{\mathcal I}$ satisfies
\[ \| v - P_{\mathcal I}v \|_{L^2(J)} \lesssim \vert \mathcal I\vert^{s+1} \|\partial_t^{s+1} v\|_{L^2(J)}, \quad
\| v - P_{\mathcal I}v \|_{H^1(J)} \lesssim \vert \mathcal I\vert^{s} \|\partial_t^{s+1} v\|_{L^2(J)} \]
for $v\in H^{s+1}(I)$. Hence, the norm corresponding operators can be bounded by
\[ \|I-P_{\mathcal I} \|_{H^{s+1}(J)\rightarrow L^2(J)} \lesssim 2^{-(s+1) L}, \quad
\|I-P_{\mathcal I} \|_{H^{s+1}(J)\rightarrow H^1(J)} \lesssim 2^{-sL}.\] 
For the projectors onto $V_L$ it follows from Thm.\ \ref{thm:griebel} that
\[ 
	\| I-P_L \|_{H_{\mathrm{mix}}^{s+1}(\Omega)\rightarrow L^2(\Omega)} \lesssim 2^{-(s+1) L} \sqrt{L}, \quad
	\| I-P_L \|_{H_{\mathrm{mix}}^{s+1}(\Omega)\rightarrow H^1(\Omega)} \lesssim 2^{-s L} \sqrt{L}.
\] 

Now, the projection error with respect to $L^2(J\times \Omega)$ can be estimated by 
\begin{align*}
&\| u - u_h \|_{L^2(J \times \Omega)} \\
&\quad = \| (I-P_{\mathcal I}) \otimes P_L \, u + P_{\mathcal I} \otimes (I-P_L) \, u + (I-P_{\mathcal I}) \otimes (I-P_L) u\|_{L^2(J \times \Omega)} \\
& \quad \le \|I-P_{\mathcal I} \|_{H^{s+1}(J)\rightarrow L^2(J)} \, \| P_L \|_{L^2(\Omega)\rightarrow L^2(\Omega)} \| u\|_{H^{s+1}(J) \otimes L^2(\Omega)} \\
& \quad \quad + \| P_{\mathcal I} \|_{L^2(J)\rightarrow L^2(J)} \, \| I-P_L \|_{H_{\mathrm{mix}}^{s+1}(\Omega)\rightarrow L^2(\Omega)} \| u\|_{L^2(J) \otimes H^{s+1}_{\mathrm{mix}}(\Omega)} \\
& \quad \quad + \| I-P_{\mathcal I} \|_{H^{s+1}(J)\rightarrow L^2(J)} \, \| I-P_L \|_{L^2(\Omega)\rightarrow L^2(\Omega)} \| u\|_{H^{s+1}(J) \otimes L^2(\Omega)}.
\end{align*}
Using the bounds of the projection operators and the fact that
\[  \| P_{\mathcal I} \|_{L^2(J)\rightarrow L^2(J)}, \, \| P_L \|_{L^2(\Omega)\rightarrow L^2(\Omega)}, \, \| I - P_L \|_{L^2(\Omega)\rightarrow L^2(\Omega)} \lesssim 1 \]
due to the orthogonality of the projectors it follows that
\begin{align*}
\| u - u_h \|_{L^2(J \times \Omega)} 
& \lesssim  2^{-(s+1) L} \| u\|_{H^{s+1}(J) \otimes L^2(\Omega)}
	+ 2^{-(s+1) L} \sqrt{L} \| u\|_{L^2(J) \otimes H^{s+1}_{\mathrm{mix}}(\Omega)} \\
& \quad + 2^{-(s+1) L} \| u\|_{H^{s+1}(J) \otimes L^2(\Omega)} \\
& \lesssim 2^{-(s+1)L} \sqrt{L} \| u\|_{\mathcal H^{s+1}}.
\end{align*}
 
For the second inequality of the theorem note that
\[ H^1(J\times \Omega) = 
	H^1(J) \otimes L^2(\Omega)
	\cap L^2(J) \otimes H^{1,0}_{\mathrm{mix}}(\Omega)
	\cap L^2(J) \otimes H^{0,1}_{\mathrm{mix}}(\Omega) \]
We estimate each contribution like in the previous case:
\begin{align*}
&\| u - u_h \|_{H^1(J) \otimes L^2(\Omega)} \\
&\quad  \le \|I-P_{\mathcal I} \|_{H^{s+1}(J)\rightarrow H^1(J)} \, \| P_L \|_{L^2(\Omega)\rightarrow L^2(\Omega)} \| u\|_{H^{s+1}(J) \otimes L^2(\Omega)} \\
& \quad \quad + \| P_{\mathcal I} \|_{L^2(J)\rightarrow H^1(J)} \, \| I-P_L \|_{H_{\mathrm{mix}}^{s+1}(\Omega)\rightarrow L^2(\Omega)} \| u\|_{L^2(J) \otimes H^{s+1}_{\mathrm{mix}}(\Omega)} \\
& \quad \quad + \| I-P_{\mathcal I} \|_{H^{s+1}(J)\rightarrow H^1(J)} \, \| I-P_L \|_{L^2(\Omega)\rightarrow L^2(\Omega)} \| u\|_{H^{s+1}(J) \otimes L^2(\Omega)} 
\end{align*}
and use the inverse inequality of $P_{\mathcal I}$ to bound $\|P_{\mathcal I} \|_{L^2(J)\rightarrow H^1(J)} \lesssim \vert \mathcal I\vert^{-1}\sim 2^{L}$ which leads to
\begin{align*}
\| u - u_h \|_{H^1(J) \otimes L^2(\Omega)} \lesssim 2^{-s L} \sqrt{L}\, \|u\|_{\mathcal H^{s+1}}.
\end{align*}
In the same manner
\begin{align*}
&\| u - u_h \|_{L^2(J) \otimes H^{1,0}_{\mathrm{mix}}(\Omega)} \\
&\quad \le \| (I-P_{\mathcal I}) \|_{H^{s+1}(J)\rightarrow L^2(J)} \, \| P_L \|_{L^2(\Omega) \rightarrow H_{\mathrm{mix}}^{1,0}(\Omega)} \| u\|_{H^{s+1}(J) \otimes L^2(\Omega)} \\
&\quad \quad + \| P_{\mathcal I} \|_{L^2(J)\rightarrow L^2(J)} \, \| I-P_L \|_{H_{\mathrm{mix}}^{s+1}(\Omega)\rightarrow H^1(\Omega)} \| u\|_{L^2(J) \otimes H^{s+1}_{\mathrm{mix}}(\Omega)} \\
&\quad \quad + \| I-P_{\mathcal I} \|_{L^2(J)\rightarrow L^2(J)} \, \| I-P_L \|_{H^{s+1}_{\mathrm{mix}}(\Omega)\rightarrow H^{1,0}_{\mathrm{mix}}(\Omega)}\| u\|_{L^2(J) \otimes H^{s+1}_{\mathrm{mix}}(\Omega)} \\
&\quad \lesssim 2^{-s L} \sqrt{L} \| u\|_{\mathcal H^{s+1}}.
\end{align*}
Here we use the inverse inequality $\| P_L \|_{L^2(\Omega) \rightarrow H^1(\Omega)} \lesssim 2^{L}$. The last norm with respect to $L^2(J) \otimes H^{0,1}_{\mathrm{mix}}(\Omega)$ is treated analogously. Summing up all contributions proves the second inequality. 
\end{proof}

\subsection{Convergence \label{sub:convergence}}

Using the approximation result from the previous section we can prove the convergence of the discrete solution.
\begin{theorem} \label{thm:convergence}
Let $u$ be a solution of the transport problem \eqref{eq:transport} and \eqref{eq:inflow} in the case of constant coefficients and $\sigma \ge 0$ and assume that $u \in \mathcal H^{s+1}$ for $0<s\le r$. Let $U$ be the solution of the discrete equations \eqref{eq:weak_eqn} for level $L$, where $\vert \mathcal I\vert  \sim h = 2^{-L}$ and $\delta \sim h$. Then 
\begin{align*}
\vvvert U-u \vvvert_{\delta} \lesssim h^{s+1/2} \, (\ln h)^{1/2} \, \|u\|_{\mathcal H^{s+1}}.
\end{align*}
\end{theorem}
\begin{proof}
The proof follows along the lines of \cite[Thm.\ 13.7]{larsson2005} and \cite{schwab2008}. Decompose the error
\[ e = U - u = (U - u_h) + (u_h - u) = \theta + \rho, \quad u_h = (P_{\mathcal I} \otimes P_L) u. \]
The streamline diffusion norm can be estimated by
\begin{align}\label{eq:error_decomp}
\vvvert e \vvvert_{\delta} \le \vvvert \theta \vvvert_{\delta} + \vvvert \rho \vvvert_{\delta}.
\end{align}
For the first term note that
\[ \vvvert \theta \vvvert_{\delta}^2 \lesssim A_{\mathcal I}^{(\delta)}(\theta, \theta) = -A_{\mathcal I}^{(\delta)}(\rho, \theta),\]
where the Galerkin orthogonality has been used. Applying \eqref{eq:AI_partint} for the right hand side, using Young's inequality and Lemma \ref{lem:Bvv_SD} gives the estimate
\begin{align*}
& \vert A_{\mathcal I}^{(\delta)}(\rho, \theta)\vert  \\
&\quad \le \sum_{j=1}^{N} \int_{I_j} 2 \delta^{-1} \|\rho\|^2 + \frac{\delta}{8} \| \partial_t \theta + \partial_\bbeta \theta \|^2 
		+ \sigma \|\rho\|^2 + \frac{\sigma}{4} \|\theta\|^2 
		 + \vert \rho\vert_{\Gammarp}^2 + \frac 1 4 \vert \theta\vert_{\Gammarp}^2 \\
	& \quad \quad \quad \quad
		 + 2 \delta \|\partial_t\rho + \partial_\bbeta \rho\|^2 + \frac{\delta}{8} \|\partial_t\theta + \partial_\bbeta \theta \|^2 
		 + 2 \delta \sigma^2 \| \rho\|^2 + \frac{\delta}{8} \| \partial_t \theta + \partial_\bbeta \theta\|^2 \, \mathrm d t\\
&\quad \quad  + \sum_{j=1}^{N} \|\rho_j\|^2 + \frac 1 4 \sum_{j=1}^{N-1} \| [\theta]_j\|^2 + \frac 1 4 \|\theta_N^-\|^2  \\
&\quad \le \frac{1}{2} A_{\mathcal I}^{(\delta)}(\theta,\theta) \\
&\quad \quad + C \big[ \sum_{j=1}^N \int_{I_j} (1+\delta^{-1}) \|\rho\|^2 + \delta \|\partial_t \rho + \partial_\bbeta \rho\|^2 + \vert \rho\vert_{\Gammarp}^2\, \mathrm dt
	+ \sum_{j=1}^N \|\rho_j\|^2 \big].
\end{align*}
Here we used the fact that $\rho$ is continuous in time. Hence it follows that
\begin{align*}
\vvvert \theta \vvvert_{\delta}^2 \lesssim \sum_{j=1}^N \int_{I_j} (1+\delta^{-1}) \|\rho\|^2 + \delta \|\partial_t \rho + \partial_\bbeta \rho\|^2 + \vert\rho\vert_{\Gammarp}^2\, \mathrm dt
	+ \sum_{j=1}^N \|\rho_j\|^2. 
\end{align*}
For the second term in \eqref{eq:error_decomp}
\begin{align*}
\vvvert \rho \vvvert_{\delta}^2 \lesssim 
\sum_{j=1}^N \int_{I_j}  \|\rho\|^2 + \delta \| \partial_t \rho + \partial_\bbeta \rho\|^2 + \vert \rho\vert^2_{\partial\Omega} \, \mathrm d t  
	+ \| \rho_0\|^2 + \| \rho_N\|^2,
\end{align*}
where again the continuity of $\rho$ was used. Applying the trace inequality on each strip $I_j\times \Omega$ gives
\begin{align*}
\| \rho_{j-1} \|^2 +  \| \rho_j \|^2 + \int_{I_j} \vert \rho\vert_{\partial \Omega}^2 \, \mathrm d t 
&\lesssim \| \rho\|_{L^2(I_j\times \Omega)} \cdot \| \rho\|_{H^1(I_j\times \Omega)} \\
&\lesssim \delta^{-1} \| \rho\|_{L^2(I_j\times \Omega)}^2 + \delta  \| \rho\|_{H^1(I_j\times \Omega)}^2.
\end{align*}
Using 
\[ \int_{I_j} \|\partial_t \rho +  \partial_\bbeta \rho\|^2\, \mathrm d t \lesssim \| \rho\|_{H^1(I_j\times \Omega)}^2,\] 
summing up and applying Thm.\ \ref{thm:approximation} gives
\begin{align*}
\vvvert e \vvvert_{\delta}^2 \lesssim 
	\big(1 + \delta^{-1}\big) \|\rho\|_{L^2(J\times \Omega)}^2 + \delta \|\rho\|_{H^1(J\times \Omega)}^2
\lesssim h_L^{2s+1} \, \ln h_L \, \| u\|_{\mathcal H^{s+1}}^2
\end{align*}
which completes the proof. 
\end{proof}

\section{Numerical Algorithm \label{sec:algorithm}}

In this section we describe how the discrete equation \eqref{eq:weak_form_j} on each time strip $I_j$, $j=1,\ldots,N$ is solved. Now we allow the coefficients $\bbeta, \, \sigma$ to be functions with sufficient regularity. The aim is to use data structures which enable the efficient representation of the sparse grid functions as well as efficient calculations. For that purpose, we use standard finite element libraries and exploit the tensor product structure of the problem. Therefore we will concentrate on coefficient functions which exhibit a tensor product structure.

\subsection{Sparse Grid Representation \label{sub:representation}}

Traditionally sparse grid spaces are built on multilevel decomposition and require special basis functions resulting in quite sophisticated operations in Galerkin schemes. This problem can be overcome by representing sparse grid functions as a combination of anisotropic full grid spaces and performing all the computations on these regular spaces. For sparse grids based on finite element spaces this approach has been used in \cite{griebel2014,harbrecht2008}, for example. This idea is closely connected with the sparse grid combination technique introduced in \cite{griebel1992}, see also \cite{garcke2012}. 

Here in this section, the focus is on the representation of sparse grid functions. In contrast to the combination technique, we will continue to solve the problem on the whole space $V_L$ and not only combining the solutions of the anisotropic spaces. However, later in Section \ref{sub:solver} we will use the combination technique as a preconditioner for the sparse grid system.

The key is the representation of the space $V_L$ in \eqref{eq:sg_space} as a sum of anisotropic full grid spaces $\Vx_{\ell_1} \otimes \Vv_{\ell_2}$ with $\ell_1 + \ell_2 = L+1$. Note that this is not a direct sum. For representing a function we use the finite element basis 
\begin{align*}
\Vi_\ell = \spn\big\{ \varphi^{(i)}_{\ell,k} \,\big\vert \, k = 1, \ldots, n^{(i)}_\ell\big\}, \quad n^{(i)}_\ell = \dim \Vi_\ell
\end{align*}
for $i=1,2$ and a basis for the polynomials on $I_j$:
\begin{align*}
\mathbb P^r(I_j,\BR) = \spn\{ \eta_s^{(j)} \,\vert \, s = 0, \ldots, r \}. 
\end{align*}
For notational simplicity we will omit the index $j$ in the following. 

Hence a function $U\in\mathbb P^r(I_j, V_L)$ can be written as
\begin{align*}
\begin{split}
U(t,\bx, \bv) = \sum_{\vert \bell\vert_1 = L+1} \, \sum_{\substack{0\le s \le r \\ \bk \le \bn_{\bell}}}
	u^{(\bell)}_{s,\bk} \, \eta_s(t) \, \varphi^{(1)}_{\ell_1,k_1}(\bx) \, \varphi^{(2)}_{\ell_2,k_2}(\bv), 
\quad \bn_{\bell} = [n^{(1)}_{\ell_1},n^{(2)}_{\ell_2}].
\end{split}
\end{align*}
The representation however is not unique since the functions 
\begin{align} \label{eq:spanningset}
\eta_s \otimes \varphi^{(1)}_{\ell_1,k_1} \otimes \varphi^{(2)}_{\ell_2,k_2}, \quad 0\le s\le r, \, \bk \le \bn_\bell, \, \vert \bell\vert_1 = L+1,  
\end{align}
are just a spanning set and not a basis. This leads to additional degrees of freedom, which are negligible for the interesting case of $d=2,3$, however.

\subsection{Discrete Equations and Matrices \label{sub:discrete_eqns}}

We will now use the spanning set \eqref{eq:spanningset} to derive a system of linear equations for the discrete equation \eqref{eq:weak_form_j} in weak form for a time strip $I_j$. Omitting again the index $j$ gives
\begin{align} \label{eq:Adelta}
A^{(\delta)} \bu = \bb^{(\delta)}, \quad \bu = \big[ \bu_{\bell}\big]_{\vert \bell\vert_1 = L+1}, \quad 
	\bu_{\bell} = \big[ u^{(\bell)}_{s,\bk} \big]_{\substack{0\le s \le r \\ \bk_i \le \bn_{\bell}}}.
\end{align}
where the blocks are given by 
\begin{align*}
A^{(\delta)}_{\bell, \bell'} & = \Big[ 
		a^{(\delta)}_j\big(\eta_{s'} \otimes \varphi^{(1)}_{\ell_1',k_1'} \otimes \varphi^{(2)}_{\ell_2',k_2'},
			\eta_{s} \otimes \varphi^{(1)}_{\ell_1,k_1} \otimes \varphi^{(2)}_{\ell_2,k_2}\big) 
	\Big]_{\substack{0\le s,s'\le r\\ {\bk \le \bn_{\bell}}, \bk' \le \bn_{\bell'}}} \\
\bb^{(\delta)}_{\bell} &=\Big[ 
		b^{(\delta)}_j\big(	\eta_{s} \otimes \varphi^{(1)}_{\ell_1,k_1} \otimes \varphi^{(2)}_{\ell_2,k_2}\big) 
	\Big]_{\substack{0\le s\le r\\ \bk \le \bn_{\bell}}}
\end{align*}
for $\vert \bell\vert_1, \vert \bell'\vert_1 = L+1$. Note that the matrix $A^{(\delta)}$ is not invertible, but the system is still solvable since the right hand side is in the range of the matrix \cite{griebel2014a}. We will later use iterative methods to solve the system which are based on the application of the matrix to a vector, see Sect. \ref{sub:solver}.

For its efficient calculation we exploit the tensor product structure for each individual block $A^{(\delta)}_{\bell,\bell'}$. In the first step we use a quadrature rule of sufficient order for the time integration on the interval $I_j = [t_{j-1}, t_j]$:
\[ \int_{I_j} f(t) \, \mathrm d t \approx \sum_{\mu=1}^{m} w_{\mu} f(\tau_\mu), \quad \tau_\mu \in [t_{j-1}, t_j]. \]
Define for $1 \le \mu \le m$
\begin{align*}
& M_\mu^{(t)} = \big[ \eta_{s'}(\tau_{\mu}) \cdot \eta_s(\tau_\mu) \big]_{s,s'}, \quad
	T_\mu^{(t)} =  \big[ \eta'_{s'}(\tau_{\mu}) \cdot \eta_s(\tau_\mu) \big]_{s,s'}, \\
& C_\mu^{(t)} = \big[ \eta_{s'}'(\tau_\mu) \cdot \eta_s'(\tau_\mu) \big]_{s,s'}, \quad
			G^{(t)} = \big[ \eta_{s'}(t_{j-1}) \cdot \eta_s(t_{j-1}) \big]_{s,s'}
\end{align*}
for $0 \le s,s' \le r$ and
\begin{align} \label{eq:matrices}
\begin{split}
&M^{(\br)}_{\bell,\bell'} = \big[ (\varphi_{\bell',\bk'}, \varphi_{\bell,\bk}) \big], \quad
	T^{(\br)}_{\mu,\bell,\bell'} = \big[ (\partial_{\bbeta(\tau_\mu, \cdot)} \varphi_{\bell',\bk'}, \varphi_{\bell,\bk}) \big], \\
&C^{(\br)}_{\mu,\bell,\bell'} = \big[ (\partial_{\bbeta(\tau_\mu, \cdot)} \varphi_{\bell',\bk'}, 
			\partial_{\bbeta(\tau_\mu, \cdot)}\varphi_{\bell,\bk}) \big], \quad
	K^{(\br)}_{\mu,\bell,\bell'} = \big[ (\sigma(\tau_\mu,\cdot) \varphi_{\bell',\bk'}, \varphi_{\bell,\bk}) \big], \\
&\tilde K^{(\br)}_{\mu,\bell,\bell'} = \big[ (\sigma(\tau_\mu,\cdot) \varphi_{\bell',\bk'}, 
		\partial_{\bbeta(\tau_\mu,\cdot)} \varphi_{\bell,\bk}) \big], \quad
G^{(\br)}_{\mu,\bell,\bell'}  = \big[ \langle \varphi_{\bell',\bk'}, \varphi_{\bell,\bk}\rangle_{\Gammax(\tau_\mu)} \big]
\end{split}
\end{align}
with indices $\bk \le \bn_{\bell}, \, \bk' \le \bn_{\bell'}$. Hence one block can be approximated by 
\begin{align*}
A^{(\delta)}_{\bell, \bell'} & \approx \bar A^{(\delta)}_{\bell, \bell'} \\
&= 
\sum_{\mu}^m w_\mu \Big[ 
	(T_\mu^{(t)} + \delta \, C^{(t)}_{\mu})\otimes M^{(\br)}_{\bell,\bell'} 
	+ \delta \, T_\mu^{(t)} \otimes \big(T^{(\br)}_{\mu,\bell,\bell'}\big)^T \\
&\quad \quad \quad
	+ \delta \, \big(T^{(t)}_{\mu}\big)^T \otimes \big(T^{(\br)}_{\mu,\bell,\bell'} +K^{(\br)}_{\mu,\bell,\bell'}\big) \\
&\quad \quad \quad
		+ M_\mu^{(t)} \otimes 
		\big(T^{(\br)}_{\mu,\bell,\bell'} + \delta\, C^{(\br)}_{\mu,\bell,\bell'} + K^{(\br)}_{\mu,\bell,\bell'}+\tilde K^{(\br)}_{\mu,\bell,\bell'} 
			-  G^{(\br)}_{\mu,\bell,\bell'}\big) \Big] \\
&\quad  
	+ G^{(t)} \otimes M^{(\br)}_{\bell,\bell'}.
\end{align*}

The matrix $M^{(\br)}_{\bell,\bell'}$ acting on the variables $\bx$ and $\bv$ has itself a tensor product structure
\begin{align*}
M^{(\br)}_{\bell,\bell'} = M^{(1)}_{\ell_1,\ell_1'} \otimes M^{(2)}_{\ell_2,\ell_2'}, \quad 
	M^{(i)}_{\ell,\ell'} = \big[ (\varphi^{(i)}_{\ell',k'},\varphi^{(i)}_{\ell,k})_{L^2(\Omega^{(i)})} \big]_{k,k'}\,,
\end{align*}
which can be exploited. In general this is not the case for the other matrices, making a $2d$-dimensional integration necessary. In the following we will therefore assume that the coefficients $\bbeta$ and $\sigma$ have a tensor product structure in the sense of
\begin{align*}
\bbeta(t,\bx,\bv) &= \sum_{i=1}^d a_i(t,\bx) \, b_i(t,\bv) \, \be_i + c_i(t,\bx) \, d_i(t,\bv) \,\, \be_{i+d}, \\
\sigma(t,\bx,\bv) &= p(t,\bx) \, q(t,\bv),
\end{align*}
where $\be_i$ is the $i$-th unit vector, or can be represented (approximately) in a short sum of such tensor products. In the first case for example
\begin{align*}
T^{(\br)}_{\mu,\bell,\bell'}  
	& = \sum_{i=1}^d \Big[ 
		\big(a_i(\tau_\mu,\cdot) \partial_{x_i} \varphi^{(1)}_{\ell_1',k_1'}, \varphi^{(1)}_{\ell_1,k_1}\big)_{L^2(\Omega^{(1)})} \cdot
		\big(b_i(\tau_\mu,\cdot) \varphi^{(2)}_{\ell_2',k_2'}, \varphi^{(2)}_{\ell_2,k_2}\big)_{L^2(\Omega^{(2)})}
	\big]_{\bk,\bk'} \\
	& + \sum_{i=1}^d \Big[ 
		\big(c_i(\tau_\mu,\cdot) \varphi^{(1)}_{\ell_1',k_1'}, \varphi^{(1)}_{\ell_1,k_1}\big)_{L^2(\Omega^{(1)})} \cdot
		\big(d_i(\tau_\mu,\cdot) \partial_{v_i} \varphi^{(2)}_{\ell_2',k_2'}, \varphi^{(2)}_{\ell_2,k_2}\big)_{L^2(\Omega^{(2)})}
	\big]_{\bk,\bk'} \\
	& = \sum_{i=1}^d \big(
			A^{(\mu)}_{i,\ell_1,\ell_1'} \otimes B^{(\mu)}_{i,\ell_2,\ell_2'} + C^{(\mu)}_{i,\ell_1,\ell_1'} \otimes D^{(\mu)}_{i,\ell_2,\ell_2'}
		\big)
\end{align*}
with appropriately defined matrices. The generalization to a short sum is straight forward. All other matrices can be treated analogously except the matrices $G^{(\br)}_{\mu,\bell,\bell'}$ for the boundary condition which in some cases may also be written in form of a tensor product, see the example in Sect \ref{sub:Lshaped}. 

Using this strategy, the calculation of the factors can can be reduced to $d$-dimen\-sional integrals over $\Omega^{(i)}$. Due to the multilevel structure of the spaces $V^{(i)}_j$ it suffices to calculate the matrices on the finest scale and then use an interpolation operator. For $\ell'<\ell$ denote by $E^{(i)}_{\ell,\ell'}$ the matrix representation of the embedding  from $V^{(i)}_{\ell'}$ in $V^{(i)}_{\ell}$. Then for example
\begin{align*}
A^{(\mu)}_{i,\ell,\ell'} = \big(E^{(i)}_{L,\ell}\big)^T \cdot A^{(\mu)}_{i,L,L} \cdot E^{(i)}_{L,\ell'}.
\end{align*}
For the calculation of $A_{i,L,L}^{(\mu)}$ classical finite element libraries can be used directly.

The application of the matrix $A^{(\delta)}$ to a vector $\bu$ can be done blockwise, i.e. where block $\bell$ is given by
\begin{align*}
\big[A^{(\delta)}\, \bu\big]_{\bell} = \sum_{\vert \bell'\vert_1=L+1} A^{(\delta)}_{\bell,\bell'} \cdot \bu_{\bell'}
\end{align*}
while the tensor product structure can be exploited for each block. However, the order of application is crucial \cite{schwab2003,zeiser2011} to prevent the need to store intermediate results on full grids. Therefore to calculate
\begin{align*}
S^{(t)} \otimes S^{(1)}_{\ell_1,\ell_1'} \otimes S^{(2)}_{\ell_2,\ell_2'} \, \bu_{\bell'}
\end{align*}
for general matrices $S^{(t)},S^{(1)},S^{(2)}$ we will use 
\begin{align*}
\big(S^{(t)}\otimes \mathrm{Id}^{(1)}_{\ell_1} \otimes \mathrm{Id}^{(2)}_{\ell_2}\big) \cdot 
	\big(\mathrm{Id}^{(t)} \otimes S^{(1)}_{\ell_1,\ell_1'} \otimes \mathrm{Id}^{(2)}_{\ell_2}\big) \cdot 
	\big(\mathrm{Id}^{(t)} \otimes \mathrm{Id}^{(1)}_{\ell_1'} \otimes S^{(2)}_{\ell_2,\ell_2'}\big) \cdot \bu_{\bell'}
\end{align*}
in the case $\ell_1'+\ell_2\le \ell_1 + \ell_2'$ and
\begin{align*}
\big(S^{(t)}\otimes \mathrm{Id}^{(1)}_{\ell_1} \otimes \mathrm{Id}^{(2)}_{\ell_2}\big) \cdot 
	\big(\mathrm{Id}^{(t)} \otimes \mathrm{Id}^{(1)}_{\ell_1} \otimes S^{(2)}_{\ell_2,\ell_2'}\big) \cdot
	\big(\mathrm{Id}^{(t)} \otimes S^{(1)}_{\ell_1,\ell_1'} \otimes \mathrm{Id}^{(2)}_{\ell_2'}\big) \cdot \bu_{\bell'}
\end{align*}
otherwise, where $\mathrm{Id}$ the identity matrix in the respective spaces. The application of such a tensor product matrix can be applied in log-linear complexity with respect to the degrees of freedom.

\subsection{Solution of Linear System of Equations \label{sub:solver}}

For the solution of \eqref{eq:Adelta} we will use an iterative method to exploit the tensor product structure of the block matrices and use the combination technique as a preconditioner. 

The sparse grid combination technique was introduced in \cite{griebel1992} (see also \cite{garcke2012}). Instead of solving the discretized equation on the sparse grid, solutions are computed on a sequence of anisotropic full grid spaces. They are then combined to form a sparse grid function. For a quite general class of symmetric bilinear forms it can be shown that this method has the same order of convergence as the sparse grid solution \cite{griebel2014}. Alternatively, it can also be used as a preconditioner for the sparse grid discretization, see \cite{griebel1994} and \cite{widmer2009} in the case of radiative transfer. 

In our non-symmetric case we follow the latter approach and use the combination technique as a preconditioner for an outer Richardson iteration, see Algorithm \ref{alg:richardson}. In order to compute the preconditioner we have to solve the transport equation 
\begin{align} \label{eq:Adeltall}
A^{(\delta)}_{\bell,\bell} \, \Delta \widehat{\bu}_\bell = \br_{\bell}
\end{align}
on small anisotropic full grid spaces with levels $L$ (line \ref{alg:solution_Lp1}) and $L-1$ (line \ref{alg:solution_L}). 

\begin{algorithm}[!ht]
\caption{Richardson iteration with combination technique preconditioner for the solution of \eqref{eq:Adelta}  \label{alg:richardson}}
\begin{algorithmic}[1]
\State $\bu \leftarrow 0$
\Repeat
	\State $\br \leftarrow A^{(\delta)} \bu - \bb^{(\delta)}$
	 \For{$\ell_1 = 1, \ldots,L$} \Comment anisotropic spaces of level $L$
	 	\State $\bell\leftarrow (\ell_1,L+1-\ell_1)$
	 	\State Solve $A^{(\delta)}_{\bell,\bell} \, \Delta \widehat{\bu}_\bell = \br_{\bell}$ \label{alg:solution_Lp1}
	 \EndFor
	\For{$\ell_1 = 1, \ldots,L-1$} \Comment  anisotropic spaces of level $L-1$
		\State $\bell \leftarrow (\ell_1,L-\ell_1)$ 
		\State $\br_{\bell} \leftarrow \big(\mathrm{Id}^{(t)} \otimes \big(E^{(1)}_{\ell_1+1,\ell_1}\big)^T \otimes \mathrm{Id}^{(2)}_{\ell_2}\big) \,\br_{\ell_1+1,\ell_2}$  \label{alg:prolong_res}
		\State Solve $A^{(\delta)}_{\bell,\bell} \, \Delta \widehat{\bu}_\bell = \br_{\bell}$ \label{alg:solution_L}
	\EndFor
	\For{$\ell_1 = 1, \ldots,L$} \Comment combination technique
		\State $\bell \leftarrow (\ell_1,L+1-\ell_1)$ 
		\State $\Delta \bu_\bell \leftarrow \begin{cases}
			\Delta \widehat{\bu}_\bell  & \ell_1=1 \\
			\Delta \widehat{\bu}_\bell - \big(E^{(1)}_{\ell_1,\ell_1-1} \otimes \mathrm{Id}^{(2)}_{\ell_2}\big)\,
	 		\Delta \widehat{\bu}_{\ell_1-1,\ell_2} & \ell_1 > 1 \end{cases}$ \label{alg:combination}
	\EndFor
	\State $\Delta \bu \leftarrow [\Delta \bu_\bell]_{\vert \bell\vert  = L+1}$
	\State $\bu \leftarrow \bu - \Delta \bu$ \Comment preconditioned Richardson step
\Until $\|\Delta \bu\|_{H^1} < \epsilon$
\end{algorithmic}
\end{algorithm}

For that purpose we use a preconditioned GMRES. As a first preliminary preconditioner for these systems, we use the inverse of
\[ \big(T^{(t)} + \delta C^{(t)} + G^{(t)}\big) \otimes M^{(1)}_{\ell_1,\ell_1} \otimes M^{(2)}_{\ell_2,\ell_2} \, , \]
which can be computed efficiently using the LU decomposition in each product space. As a consequence the tensor product structure (see Sect. \ref{sub:discrete_eqns}) can be used for all calculations and neither the full matrix $A^{(\delta)}$ \eqref{eq:Adelta} nor the smaller matrices $A^{(\delta)}_{\bell,\bell}$ \eqref{eq:Adeltall} have to be assembled. 

Alternatively, we can also use other preconditioners like ILU or direct sparse solvers. For that purpose, however, the matrices $A^{(\delta)}_{\bell,\bell}$ for the small anisotropic full grid spaces have to be assembled. For the future it would be beneficial to devise schemes similar to \cite{griebel2014a} for the case of symmetric bilinear forms or use more advanced preconditioned iterative schemes for the anisotropic full grids \cite{reisinger2004}.

Using the combination technique we compute the preconditioned residuum in line \ref{alg:combination}. Note that we used the embedding to represent the solution on spaces with level $\vert \bell\vert_1=L+1$. For interpolation such a scheme can be shown to give the exact result \cite{garcke2012}.  

The iteration is stopped, if the norm of preconditioned residuum is small enough, i.e. $\| \Delta \bu \|_{H^1(J\times \Omega)} < \epsilon$, for some given tolerance $\epsilon$.

\section{Numerical Experiments \label{sec:experiments}}

In this section we present numerical examples to study the method described in Sect.\ \ref{sec:algorithm}. The implementation is based on the finite element library \verb|MFEM| \cite{mfem} and uses its Python wrapper \verb|PyMFEM|\footnote{\url{https://github.com/mfem/PyMFEM}}. Due to the regular data structure of the combination technique standard numerical routines from \verb|numpy|\footnote{\url{https://numpy.org/}} and \verb|scipy|\footnote{\url{https://scipy.org/}} were used. The \verb|cupy|\footnote{\url{https://cupy.dev/}} library was used as a replacement for the numerical libraries to perform the calculations on a GPU with only minor code changes. The computations were carried out on a desktop computer (i9 7900, 128 GB RAM, GTX 1080). 

\subsection{Linear advection with constant coefficients}

As the first example we consider the linear advection equation with constant coefficients in $d+d$ dimensions
\begin{align} \label{eq:constant_coeff}
\begin{split}
\partial_t u + \bone \cdot \nabla_{\bx} u + \bone \cdot \nabla_{\bv} u &= 0, \quad \bx, \bv \in [0,1]^d \\
u(0, \bx, \bv) &= \sin\big(2 \pi \, \sum_{i=1}^d (x_i + v_i) \big)
\end{split}
\end{align}
with periodic boundary conditions. The solution is given by
\begin{align*}
u(t, \bx, \bv) = \sin\Big(2 \pi \, \big(\sum_{i=1}^d (x_i + v_i) - 2 d t\big) \Big),
\end{align*}
which is periodic in $t$ for $T=1/(2d)$.

The mesh on the coarsest scale on each $\Omegai$ consists of $4^d$ $d$-cubes of equal size and is uniformly refined for finer scales. The finite element functions are continuous piecewise polynomials of order $r$ for $r=1,2$. The streamline diffusion parameter is chosen as the edge length on the finest scale. The equation is solved on one period $T$ using discontinuous polynomials of the same order $r$ as in the spatial discretization and $2^{L+1}$ steps.

The numerical solution is compared to the interpolation of the analytical solution in $V_{L+1,\mathcal I}$ and polynomial order $r+1$. The streamline diffusion norm of the difference is computed and reported in Table \ref{tab:constant_coeff}. There the order is estimated from the errors of two successive runs.  The convergence is also depicted in Fig. \ref{fig:constant_coeff}. The numerical results confirm the convergence order of $h^{r+1/2}$ from Thm.\ \ref{thm:convergence}.

\begin{table}
\caption{Error with respect to the streamline diffusion norm $\vvvert\cdot\vvvert_\delta$ and estimated order of the numerical solution of \eqref{eq:constant_coeff} on $d+d$ dimensions after one period for sparse grid spaces with levels $L$ and order $r$ of the ansatz functions; $h$ denotes the grid size of the finest level}
\label{tab:constant_coeff}
\begin{center}
\begin{tabular}{ccrrrrrr}
\toprule
$L$ & $h$ & dof & error & order & dof & error & order \\ \midrule
$d=1$ & & $r=1$ & & &$r=2$& &  \\\hline 
1 & 2.50e-01 & 16 & 8.68e-01 & & 64 & 3.09e-01 & \\
2 & 1.25e-01 & 48 & 5.25e-01 & 0.72& 192 & 5.80e-02 & 2.42\\
3 & 6.25e-02 & 128 & 2.03e-01 & 1.37& 512 & 1.04e-02 & 2.48\\
4 & 3.12e-02 & 320 & 7.20e-02 & 1.49& 1280 & 1.84e-03 & 2.50\\
5 & 1.56e-02 & 768 & 2.55e-02 & 1.50& 3072 & 3.25e-04 & 2.50\\
6 & 7.81e-03 & 1792 & 8.99e-03 & 1.50& 7168 & 5.74e-05 & 2.50\\ \midrule
$d=2$ & & $r=1$ & & &$r=2$& &  \\\midrule
1 & 2.50e-01 & 256 & 8.14e-01 & & 4096 & 3.92e-01 & \\
2 & 1.25e-01 & 1792 & 5.98e-01 & 0.45& 28672 & 7.47e-02 & 2.39\\
3 & 6.25e-02 & 10240 & 2.64e-01 & 1.18& 163840 & 1.34e-02 & 2.48\\
4 & 3.12e-02 & 53248 & 9.58e-02 & 1.46& 851968 & 2.37e-03 & 2.50\\
5 & 1.56e-02 & 262144 & 3.39e-02 & 1.50& 4194304 & 4.19e-04 & 2.50\\
6 & 7.81e-03 & 1245184 & 1.19e-02 & 1.51& 19922944 & 7.41e-05 & 2.50\\
\midrule
$d=3$ & & $r=1$ & & &$r=2$& &  \\\midrule 
1 & 2.50e-01 & 4096 & 7.36e-01 & & 262144 & 4.59e-01 & \\
2 & 1.25e-01 & 61440 & 6.17e-01 & 0.25& 3932160 & 8.94e-02 & 2.36\\
3 & 6.25e-02 & 720896 & 3.13e-01 & 0.98& 46137344 & 1.60e-02 & 2.48\\
4 & 3.12e-02 & 7602176 & 1.17e-01 & 1.42&&  &\\ \botrule
\end{tabular}
\end{center}
\end{table}

\begin{figure}
\begin{center}
\includegraphics[scale=0.7]{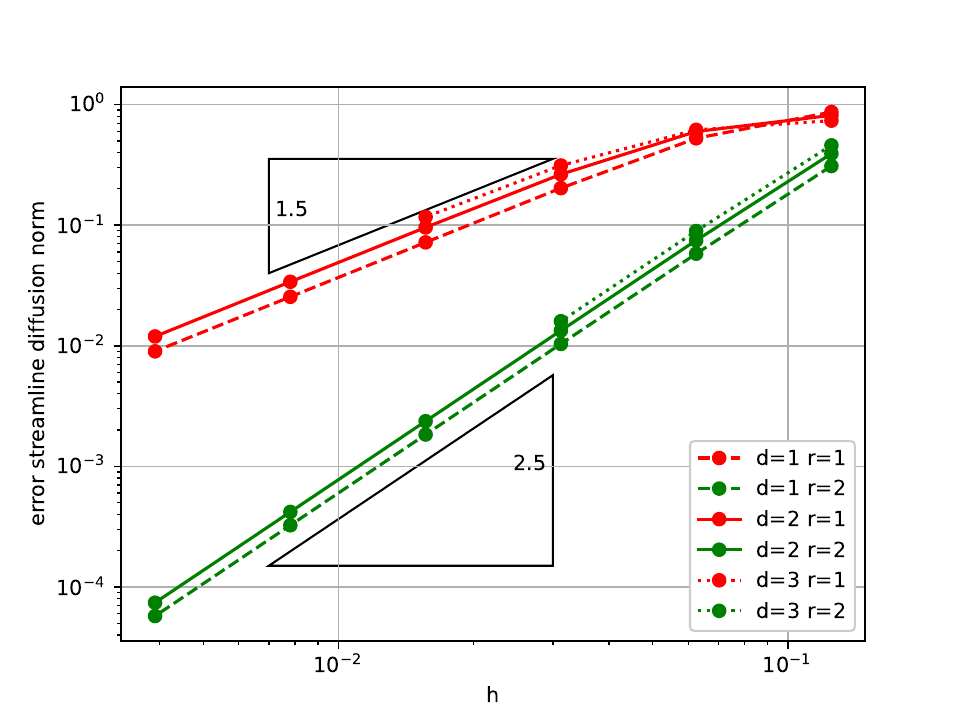}
\end{center}
\caption{Error with respect to the streamline diffusion norm $\vvvert\cdot\vvvert_\delta$ of the numerical solution of \eqref{eq:constant_coeff} on $d+d$ dimensions after one period for sparse grid spaces with order $r$ of the ansatz functions; $h$ denotes the grid size of the finest level}
\label{fig:constant_coeff}
\end{figure}

\subsection{L-shaped Spatial Domain \label{sub:Lshaped}}

In the next example we solve the kinetic equation
\begin{align} \label{eq:lshaped}
\partial_t u + \bv \cdot \nabla_\bx u = 0, \quad [0,1] \times \Omega
\end{align}
on a L-shaped spatial domain  
\begin{align*}
\Omega = \Omegax \times \Omegav, \quad 
	\Omegax = (-1,1)^2 \setminus [0,1]^2, \quad
	\Omegav = [-2,0] \times [0,2], 
\end{align*}
see Fig. \ref{fig:Lshapedmesh}. This demonstrates the use of the method for non-rectangular domains, boundary conditions, and discontinuous solutions. We choose zero inflow boundary condition and 
\begin{align*}
u_0(\bx,\bv) = 
	\psi\Big(\frac{x_1-0.5}{0.25}\Big) \cdot
	\psi\Big(\frac{x_2+0.5}{0.25}\Big) \cdot
	\psi\Big(\frac{v_1+1}{0.25}\Big) \cdot
	\psi\Big(\frac{v_2-1}{0.25}\Big)
\end{align*}
as the initial condition. Here $\psi$ is a $C^1$ function centered at $0$ and supported in $[-1,1]$:
\begin{align*}
\psi(\xi) = \begin{cases}
(\vert \xi\vert -1)^2 \cdot (2\vert \xi\vert +1) & \vert \xi\vert  \le 1 \\
0 & \vert \xi\vert  > 1
\end{cases}
\end{align*} 
For the assessment of the numerical results we will concentrate on the spatial distribution
\begin{align} \label{eq:spatial_dist}
\rho(t,\bx) = \int_{\Omegav} u(t,\bx,\bv) \, \mathrm d \bv.
\end{align}

The initial distribution is supported in $[0.25,0.75] \times [-0.75,-0.25]$. It is transported mainly in the direction $[-1,1]^T$ such that it is completely located in $[-1,0]\times[0,1]$ at the final time $t=1$, see Fig. \ref{fig:Lshapedmesh}. By the method of characteristics  the solution is given by 
\begin{align} \label{eq:analytic_lshaped}
u(1,\bx,\bv) = \begin{cases}
u_0(\bx-\bv,\bv) & v_1 x_2 > v_2 x_1 \\
0  & \textnormal{otherwise}
\end{cases}
\end{align}
for $\bx \in [-1,0]\times[0,1], \, \bv \in [-1.25,-0.75]\times[0.75,1.25]$. Outside this domain it is zero. Here, the L-shaped domain boundary leads to a discontinuous distribution.

\begin{figure}
\begin{center}
\includegraphics[trim=150 520 300 120, clip,width=0.4\textwidth]{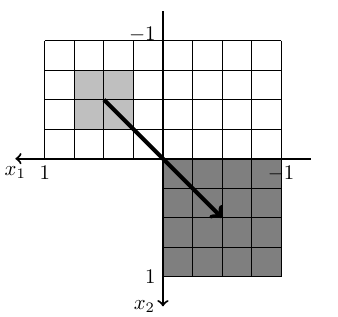}
\includegraphics[width=0.45\textwidth]{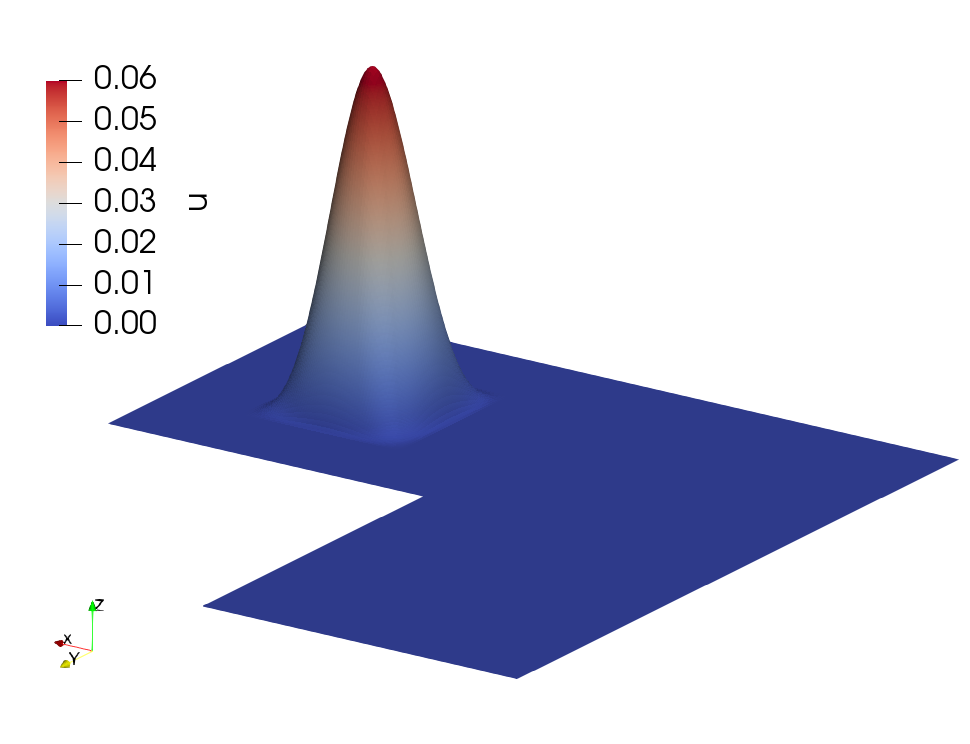}
\end{center}
\caption{Left: spatial mesh on the coarsest scale for domain $\Omegax$. The initial distribution is supported in the light gray area and is transported mainly along the arrow. It is completely contained in the dark area at $t=1$; right: initial spatial distribution $\rho(0,\bx)$  }
\label{fig:Lshapedmesh}
\end{figure}

The matrices $G^{(\br)}_{\mu,\bell,\bell'}$ from \eqref{eq:matrices} corresponding to the boundary condition can be represented as a short sum of tensor products by partitioning the inflow boundary. For example the inflow at $x_1=-1$ with normal $\bn = [-1,0,0,0]^T$ is given by
\begin{align*}
\{-1\} \times [-1,1] \times \big\{ \bv \in \Omegav \,\big\vert\, -v_1 < 0 \big\}.
\end{align*}
Hence its contribution to the element $\bk,\bk'$ of $G^{(\br)}_{\mu,\bell,\bell'}$ can be calculated by
\begin{align*}
\begin{split}
\int_{-1}^1 \int_{\Omegav} &\varphi^{(1)}_{\ell_1',k_1'}(-1, x_2) \cdot \varphi^{(1)}_{\ell_1,k_1}(-1, x_2) 	\\
&\quad \cdot \chi_{v_1>0}(\bv) \cdot (-v_1) \cdot \varphi^{(2)}_{\ell_2',k_2'}(\bv) \cdot \varphi^{(2)}_{\ell_2,k_2}(\bv) 
	\, \mathrm d \bv \, \mathrm d x_2
\end{split}
\end{align*}
with the characteristic function $\chi$. This integral as well as the integrals over other sections of the inflow boundary can be factorized leading to a sum of tensor product operators for $G^{(\br)}_{\mu,\bell,\bell'}$.

The coarsest mesh for the spatial domain $\Omegax$ is depicted in Fig.\ \ref{fig:Lshapedmesh} and the domain $\Omegav$ is discretized by an $8 \times 8$ regular grid. The equation is simulated for $t\in[0,1]$ with linear elements ($r=1$). 

Fig.\ \ref{fig:Lshaped} (left) shows the resulting spatial distribution for $t=1$ and $L=5$. In order to assess the numerical result it is compared to the distribution computed from the analytical solution \eqref{eq:analytic_lshaped} (right). The integration with respect to $\bv$ was carried out numerically. The results for levels up to $L=6$ are given in Table \ref{tab:Lshaped}.

\begin{figure}
\begin{center}
\includegraphics[width=0.45\textwidth]{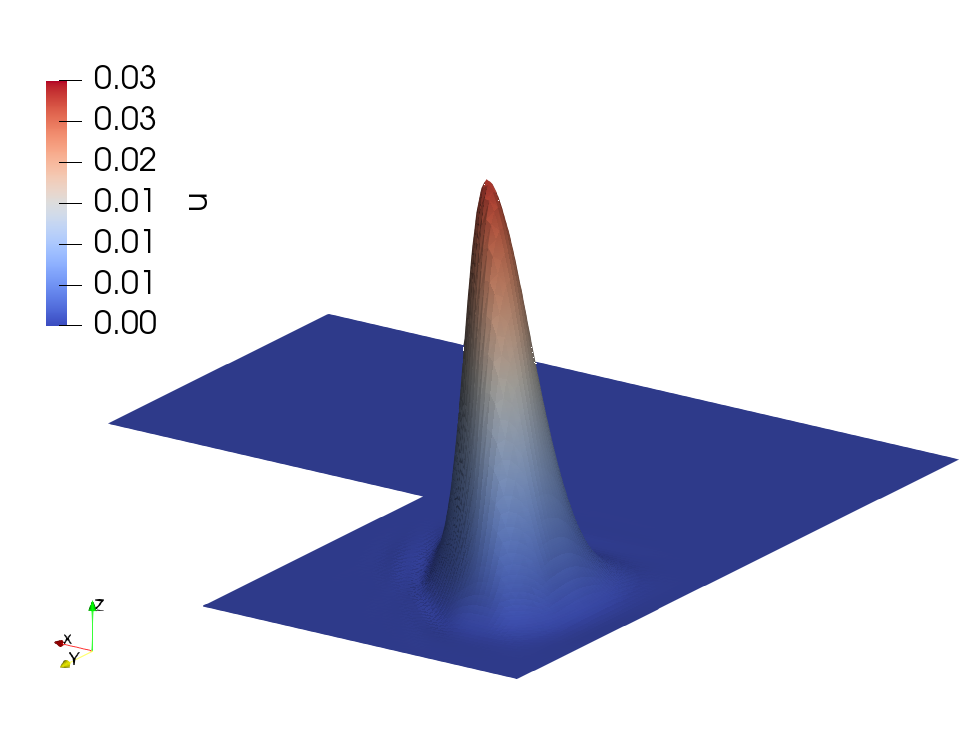}
\includegraphics[width=0.45\textwidth]{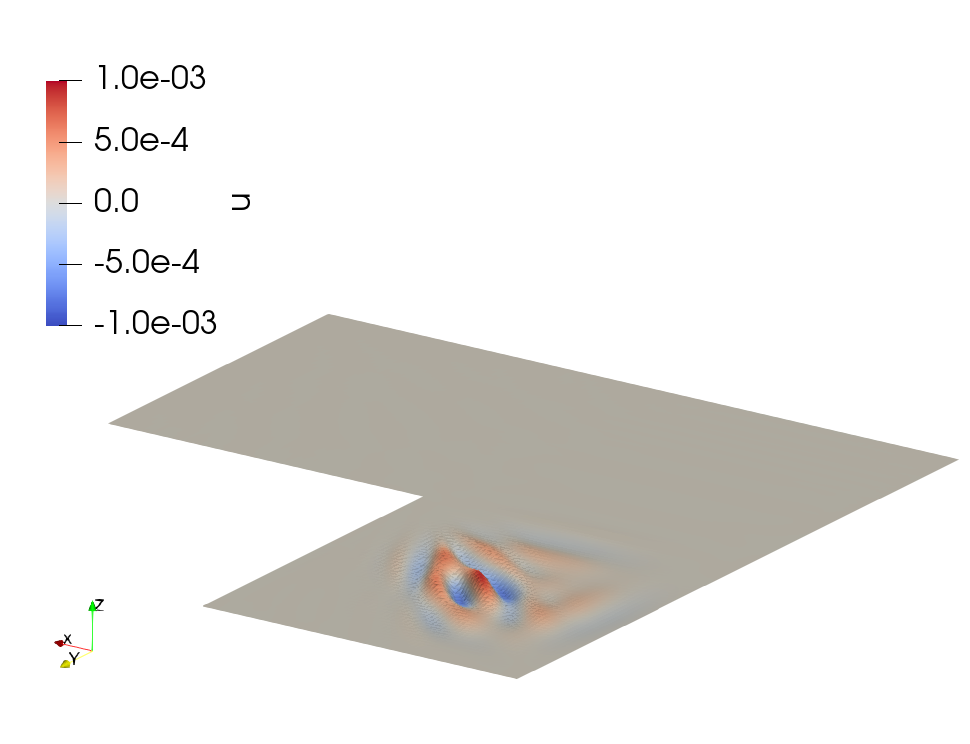}
\end{center}
\caption{Left: Computed spatial distribution $\rho(1,\bx)$ of \eqref{eq:lshaped} for $L=5$ and $r=1$; right: error of the computed spatial distribution \label{fig:Lshaped}}
\end{figure}

Although the solution $u$ is discontinuous the numerical distribution $\rho$ does not show oscillations thanks to the streamline diffusion. The theoretical results do not cover this case and it is too early to assess the order of convergence numerically. However the method seems to perform reasonable in this case.

\begin{table}[ht]
\caption{Relative error of the spatial distribution \eqref{eq:spatial_dist} with respect to the $L_2$-norm and estimated order of the numerical solution of \eqref{eq:lshaped} at $t=1$ for sparse grid spaces with levels $L$ and linear elements; $h$ denotes the grid size of the finest level}
\label{tab:Lshaped}
{\begin{center}
\begin{tabular}{crrrrr} \toprule
$L$ & 2 & 3 & 4 & 5 & 6\\\midrule
$h$ & 6.25e-02 & 3.12e-02 & 1.56e-02 & 7.81e-03 & 3.91e-03\\
rel. err. & 5.89e-01 & 2.68e-01 & 6.56e-02 & 2.58e-02 & 1.28e-02\\
order &  & 1.14 & 2.03 & 1.34 & 1.01\\\botrule
\end{tabular}
\end{center}}
\end{table}

\subsection{Vlasov-Poisson Equation}

In plasma physics the particle density function $f(t,\bx,\bv)$ of electrons in a constant background ion density interacting with electrostatic fields ignoring collisions may be described by the Vlasov-Poisson equations
\begin{align*}
\begin{split}
&\partial_t f + \bv \cdot \nabla_\bx f - \bE(f) \cdot \nabla_\bv f = 0 \\
&	\nabla_\bx \cdot \bE = -\int f(t,\bx,\bv) \, \mathrm d \bv + 1, \quad \nabla \times \bE = \mathbf 0
\end{split}
\end{align*}
supplemented by appropriate boundary conditions \cite{einkemmer2018a}. The first equation fits in the framework of the transport equation \eqref{eq:transport} except for the non-linearity due to the coupling with the electrical field. Nevertheless we may still use the method in a fixed-point iteration where the transport and field equation are solved alternately. 

Given some approximation of the electrical field $\bE^{(k)}$ we may compute the density function $f^{(k+1)}$ by solving
\begin{align} \label{eq:E_f}
\partial_t f^{(k+1)} + \bv \cdot \nabla_\bx f^{(k+1)} - \bE^{(k)}(t,\bx) \cdot \nabla_\bv f^{(k+1)} = 0.
\end{align}
using the method for linear transport equation presented here. Now given $f^{(k+1)}$ we compute the resulting electrical field by solving the Poisson equation
\begin{align} \label{eq:poisson}
-\Delta_\bx \phi = -\int f^{(k+1)}(t,\bx,\bv) \, \mathrm d \bv + 1
\end{align}
and letting $\bE^{(k+1)} = -\nabla_\bx \phi$.

Hence to solve the coupled equation in a time strip $I_j=[t_{j-1}, t_j]$ we alternate the two steps. Starting with $f^{(0)}$ as the final density $f(t_{j-1}^-)$ of the last step and the resulting electrical field $\bE^{(0)}$ for all $t\in I_j$ we iterate until
\begin{align*}
\| \bE^{(k+1)} - \bE^{(k)} \|_{L^\infty(I_j\times \Omegax)} 
	+ \| f^{(k+1)} - f^{(k)} \|_{H^1(I_j\times \Omegax\times \Omegav)} < \epsilon
\end{align*}
for some tolerance $\epsilon$.

Note that \eqref{eq:poisson} has to be solved only at certain quadrature points $\tau_\mu \in I_j$, see Sect.\ \ref{sub:discrete_eqns}, and that the integral of $f$ on the right hand side can be represented by a finite element function on the finest grid of the space domain $\Omegax$. Its solution is computed by the finite element method on the same grid. The preconditioner needed to solve \eqref{eq:E_f} is updated if the rate of convergence decreases significantly. 

As the first test we consider the classic Landau damping. In the 1+1-dimensional case we use the initial condition
\[ f(0,x,v) = f_0(x,v) = \frac{1}{\sqrt{2 \pi}} \mathrm e^{-v^2/2} \big(1 + \alpha \cos(k x) \big) \]
and parameters $\alpha = 10^{-2}$, $k=1/2$ on a periodic domain $(0, 4 \pi) \times (-v_{\mathrm{max}}, v_{\mathrm{max}})$ with $v_{\mathrm{max}}$ big enough \cite{einkemmer2018a}. We choose $v_{\mathrm{max}}=6$. 

We simulate the system in the time interval $[0,50]$. On level $L=1$ each domain, i.\ e.\ space and velocity, is divided into four equally sized subintervals which are refined globally for higher levels. For $L=5$ one hundred times steps are performed, which are doubled for each increase in level. For the computation of the electrical field the midpoint rule is applied. A direct solver computes the resulting systems of linear equations for the transport as well as the potential equation.

Linear analysis shows that the electric field decays with a rate of $\gamma \approx 0.153$ \cite{pham2013}. In Fig.\ \ref{fig:landau1d} the electrical energy 
\begin{align*}
\frac 1 2 \, \int_{\Omegax} \| \bE(t,\bx) \|^2 \, \mathrm d \bx
\end{align*}
exhibits the analytical decay rate up to $t=25$ or $t=30$ depending on the level $L$. Furthermore the invariants particle number, total energy and entropy, i.e.\
\begin{align} \label{eq:invariants}
\begin{split}
&\int_{\Omega} f(t,\bx,\bv) \, \mathrm d \bx \, \mathrm d \bv, \quad
	\frac 1 2 \int_{\Omega} \vert \bv\vert^2 \, f(t,\bx,\bv) \, \mathrm d \bx \, \mathrm d \bv 
		+ \frac 1 2 \int_{\Omega^{(\bx)}} \vert \bE(t,\bx)\vert^2 \, \mathrm d \bx, \quad \\
&\int_{\Omega} \vert f(t,\bx,\bv)\vert^2 \, \mathrm d \bx \, \mathrm d \bv
\end{split}
\end{align}
are shown. We see that the mass is almost conserved whereas the energy and entropy only up to a small error. 

\begin{figure}
\begin{center}
\includegraphics[width=0.95\textwidth]{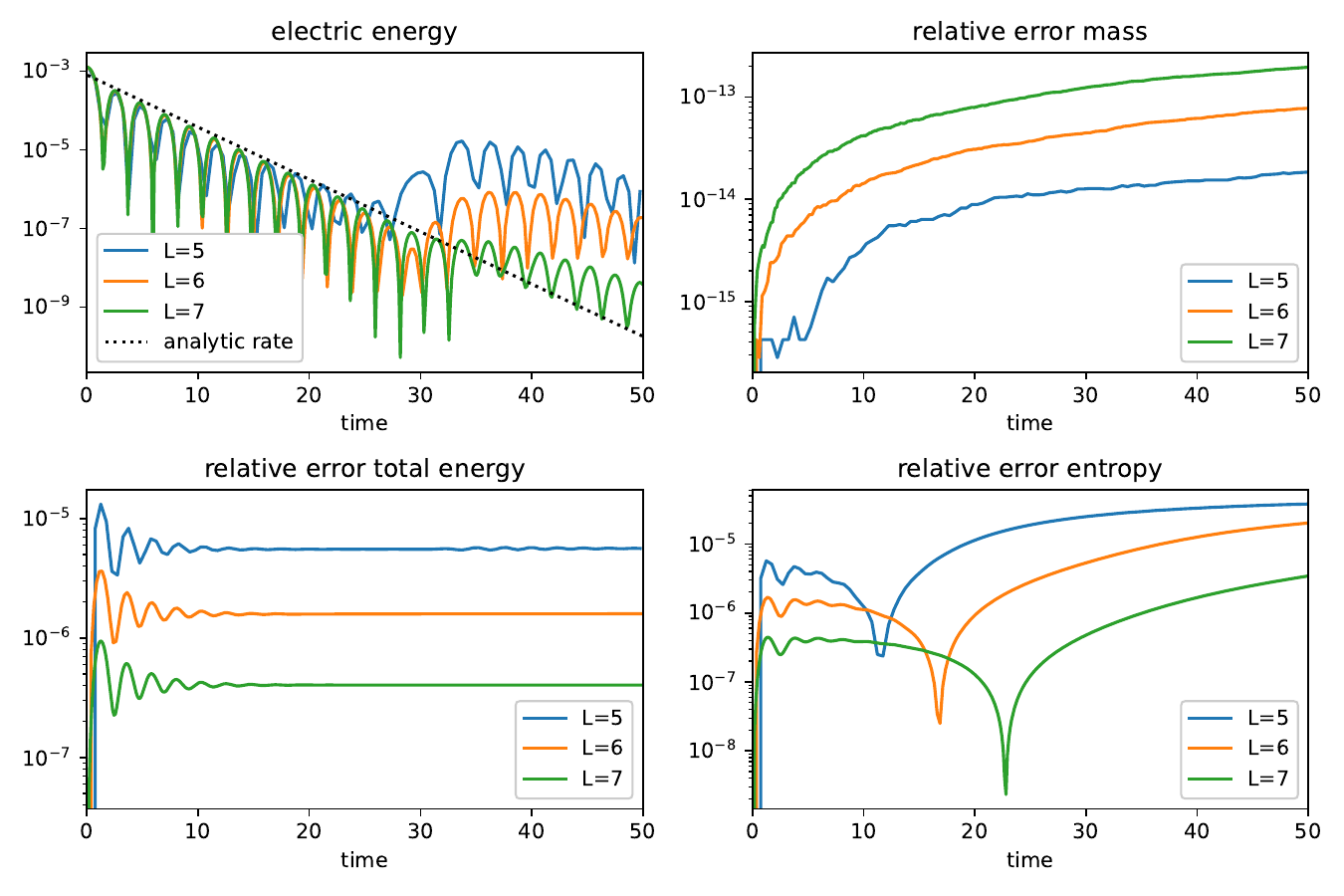}
\end{center}
\caption{Simulation results of the 1+1-dimensional Landau damping for linear elements and different levels $L$; the electric energy including the analytical decay rate and relative error of the invariants \eqref{eq:invariants}}
\label{fig:landau1d}
\end{figure}

In the 2+2-dimensional setting we use $\Omega = (0, 4 \pi)^2 \times (-6, 6)^2$ and
\begin{align*}
f(0,\bx,\bv) = \frac{1}{2\pi} \mathrm e^{-\vert \bv\vert^2/2}\, \big( 1 + \alpha \cos(k x_1) + \alpha \cos(k x_2)\big), \quad 
\alpha = 10^{-2}, \, k = \frac 1 2.
\end{align*}
Again we simulate the system with the same settings as in the 1+1-dimensional case, where the coarsest meshes now consist of four by four squares. The results are shown in Fig.\ \ref{fig:landau2d}. Qualitatively the results of the 1+1-dimensional setting are reproduced with slightly higher errors. 
\begin{figure}
\begin{center}
\includegraphics[width=0.95\textwidth]{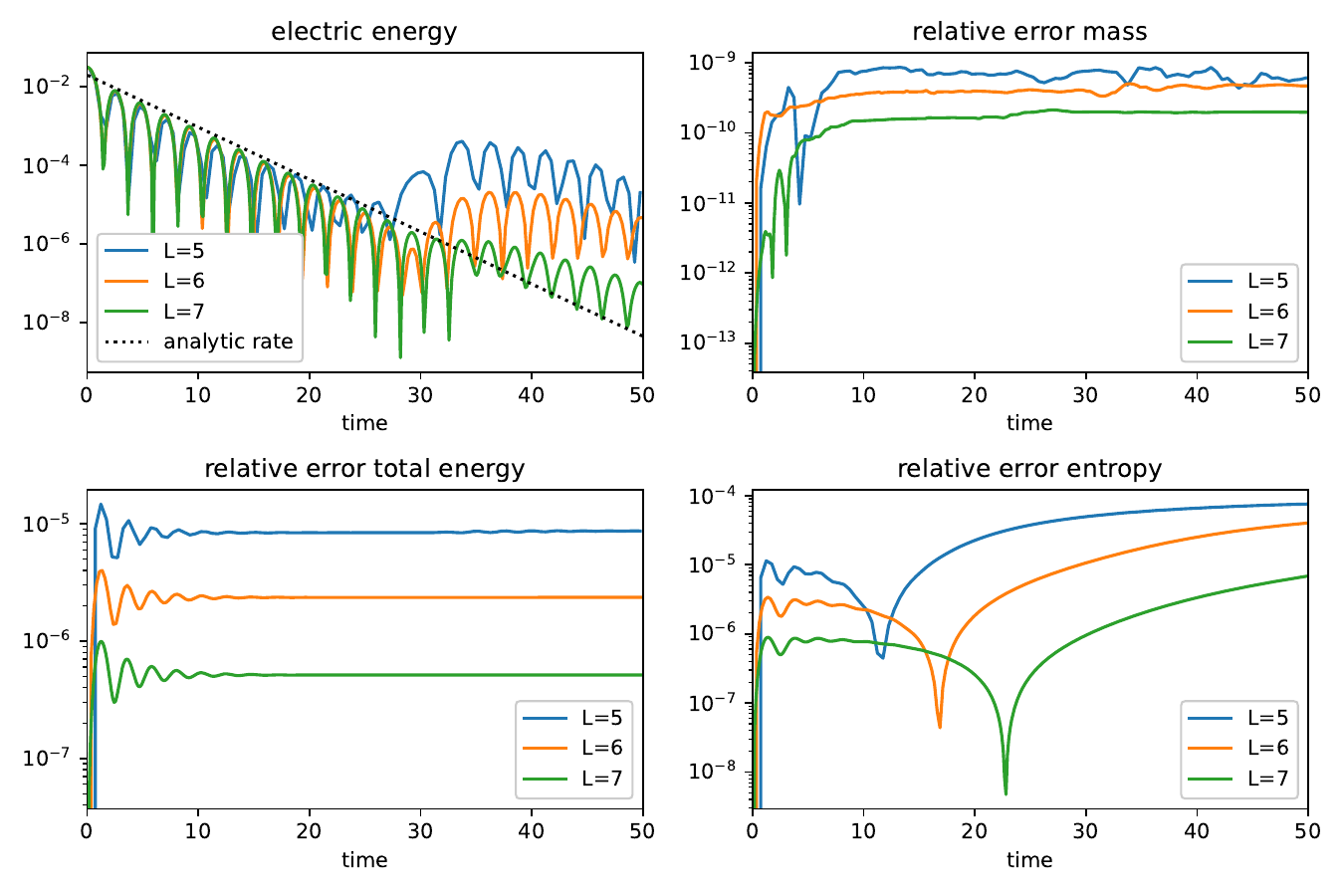}
\end{center}
\caption{Simulation results of the 2+2-dimensional Landau damping for linear elements and different levels $L$; the electric energy including the analytical decay rate and relative error of the invariants \eqref{eq:invariants}  }
\label{fig:landau2d}
\end{figure}

The second example is the two-stream instability, where two beams propagate in opposite directions. Here, small perturbations in the initial condition lead to an exponential increase in electrical energy and a subsequent saturation.

In the 1+1-dimensional case we use the periodic domain $\Omega = (0,10\pi) \times (-9,9)$ and the initial condition
\[ f(0,x,v) = \frac{1}{2\sqrt{2\pi}} \left( \mathrm e^{-(v-v_0)^2/2} + \mathrm e^{-(v+v_0)^2/2} \right) \big(1 + \alpha \cos(k x)\big),\]
where $\alpha = 10^{-3}$, $k=1/5$ and $v_0=2.4$. The results are shown in Fig.\ \ref{fig:landau2s_1d}.

For all levels used, one observes an almost identical exponential increase of the electrical energy and a saturation where the energy remains almost constant. Regarding the physical invariants, the mass is almost conserved, while the total energy and especially the entropy show a strong increase. Comparing these results with the Landau damping, we see that this is a more challenging problem. We note that the increase in the error occurs during the time when the electrical energy increases exponentially and then remains almost constant.  

\begin{figure}
\begin{center}
\includegraphics[width=0.95\textwidth]{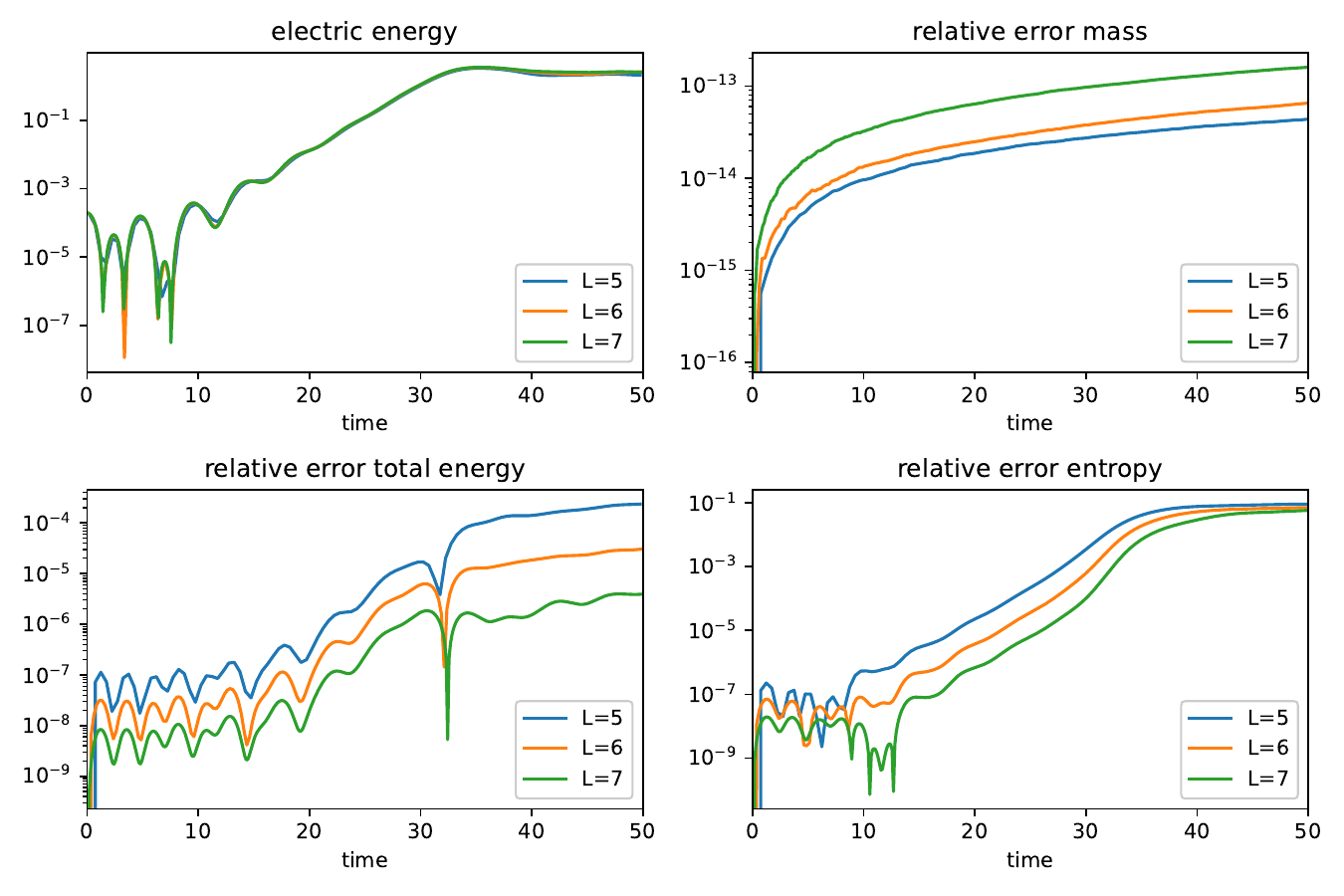} 
\end{center}
\caption{Simulation results of the 1+1-dimensional two-stream instability for linear elements and different levels $L$}
\label{fig:landau2s_1d}
\end{figure}

Now turning to the 2+2-dimensional case we let $\Omega = (0,10\pi)^2 \times (-9,9)^2$ with periodic boundary condition and
\begin{align*}
f(0,\bx,\bv) = \frac{1}{8\pi} \big( 1 + \alpha \cos(k x_1) + \alpha \cos(k x_2)\big) \prod_{k=1}^2 \left( \mathrm e^{-(v_i-v_0)^2/2} + \mathrm e^{-(v_i+v_0)^2/2} \right),
\end{align*}
where again $\alpha = 10^{-3}$, $k=1/5$ and $v_0=2.4$. The results are shown in Fig.\ \ref{fig:landau2s_2d}.

The energy as well as the invariants show almost identical behavior as in the 1+1-dimensional case. Only the error of mass increases by several orders of magnitude, but is still small. The error of the invariants improves with increasing level $L$.

\begin{figure}
\begin{center}
\includegraphics[width=0.95\textwidth]{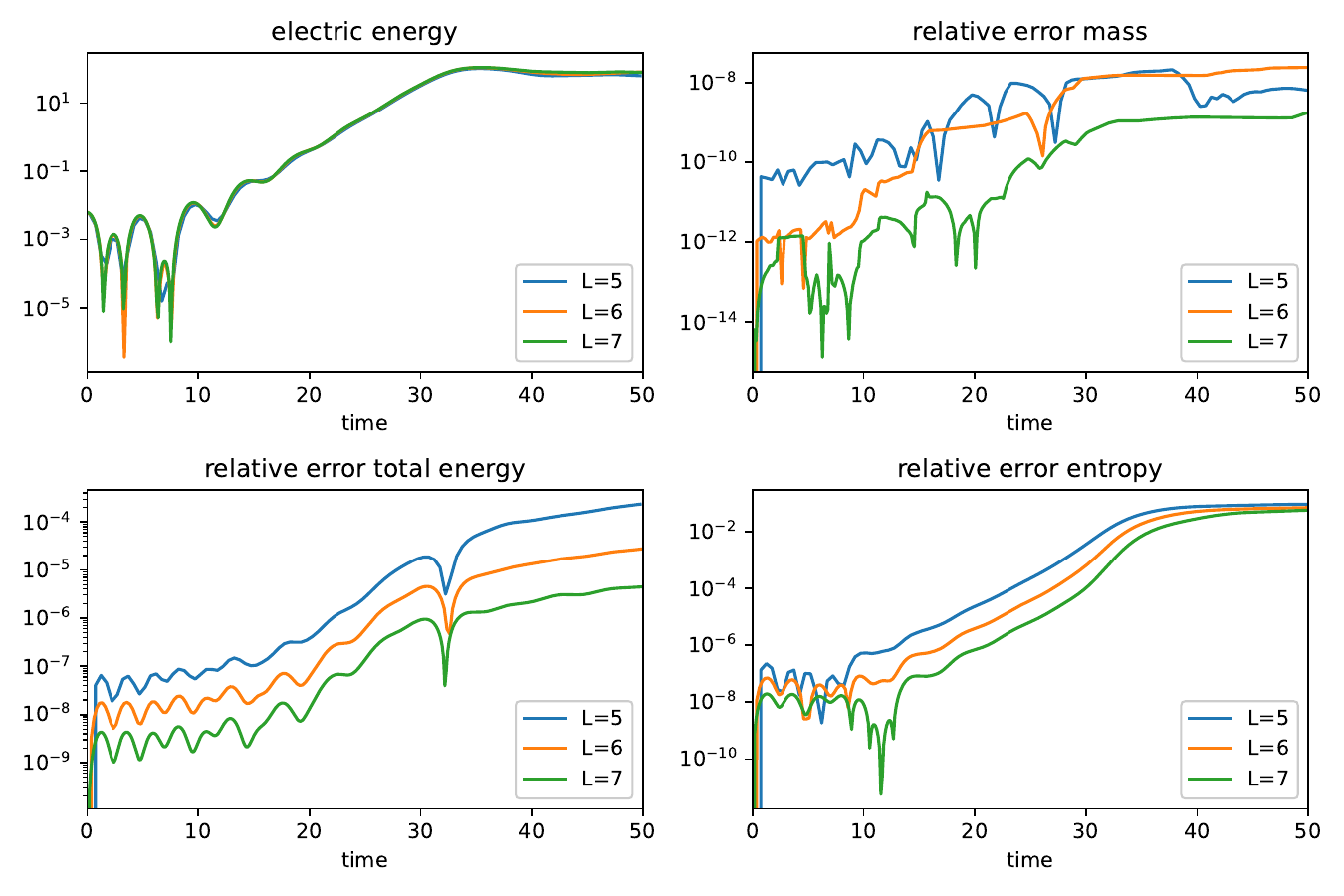}
\end{center}
\caption{Simulation results of the 2+2-dimensional two-stream instability for linear elements and different levels $L$}
\label{fig:landau2s_2d}
\end{figure}

In conclusion both tests show that the method can be used also in the case of nonlinear problems. However, the execution times are quite high: for the 2+2-dimensional two-stream instability the simulation takes a little less than three hours for level $L=5$, two days for $L=6$, and two weeks for $L=7$. Note that the code is still prototypical and not optimized for speed. For example, the problems \eqref{eq:Adeltall} arising in the computation of the combination technique preconditioner were solved sequentially, although parallelization is possible. Furthermore, the solution of these systems is the most time-consuming part because of the lack of an efficient preconditioner. In contrast, the Richardson iteration as well as the fixed-point iteration of the density and the electric field require only a few iterations.

\section{Outlook}

In this paper we present a method which is capable of computing the solution of transport equations in moderately complex geometrical domains in up to 3+3-dimensional phase space. By representing sparse grid functions as a linear combination of grid functions in anisotropic full grid spaces, traditional finite element libraries can be used. The resulting regular data structures facilitate parallelization and the use of GPU acceleration. The resulting linear equation can be solved iteratively using a combination technique preconditioner. By using a fixed-point iteration, this can also be used to compute the solution of nonlinear Vlasov-Poisson equations. 

We aim to use the method for more complex geometries -- for example, to simulate a plasma in an accelerator \cite{kessler2020}. For such realistic problems it will be necessary to further parallelize the method. Again, the combination technique naturally leads to independent problems where each problem is a transport equation on anisotropic full grids which itself can be solved on multiple machines. It remains to further investigate more efficient preconditioners for these systems.

The datasets generated during and/or analyzed in the current study are available from the corresponding author on reasonable request.

\subsection*{Acknowledgement}

This version of the article has been accepted for publication, after peer review but is not the Version of Record and does not refect post-acceptance improvements, or any corrections. The Version of Record is available online at: \\
\url{https://doi.org/10.1007/s42985-023-00250-2}.

\bibliography{SG_DG_SD}

\end{document}